\documentclass{article}
\usepackage{authblk}

\usepackage{subfiles}
\usepackage[latin1]{inputenc} 
\usepackage[top=1in, bottom=1in, left=1in, right=1in,footskip=.5in]{geometry}
\usepackage{amssymb,amsmath,amsthm,amscd,epsf,latexsym,verbatim,graphicx,amsfonts} 
\input epsf.tex
\usepackage{csquotes}
\usepackage[english]{babel}
\usepackage{enumerate}
\usepackage{enumitem}
\usepackage{mathrsfs}
\usepackage{tikz-cd}
\usepackage{hyperref}
\usepackage{float}

\usepackage[notref,notcite,final]{showkeys}
\usepackage{todonotes}

\newcounter{Theorem}[section]

\newtheorem{problem}[subsection]{Problem}
\newtheorem{theorem}[Theorem]{Theorem}

\newtheorem{definition}[Theorem]{Definition}
\newtheorem{corollary}[Theorem]{Corollary}

\newtheorem{lemma}[Theorem]{Lemma}

\numberwithin{Theorem}{section}

\newcommand{\bbL}{\mathbb{L}}
\newcommand{\D}{\mathcal{D}}
\newcommand{\Dm}{\mathcal{D}^{m}}

\newcommand{\mo}[1]{ \ (\!\!\!\!\mod #1 ) }

\tikzset{
  rootvertex/.style={fill=black, circle,inner sep=1pt},	 
	 whiteroot/.style={fill=white, draw,rectangle,inner sep=2pt},	 
	 invertex/.style={fill=black,circle,inner sep=1pt},
	 oldroot/.style={fill=black,rectangle,inner sep=1.4pt},	 
	 rleaf/.style={fill=red,draw=red,circle,inner sep=1.5pt},	
  wleaf/.style={fill=black,rectangle,inner sep=2pt},
	 centro/.style={fill=blue,draw,circle,inner sep=1.5pt},	 
	 pat1/.style={fill=purple!20},	 
	 pat2/.style={fill=blue!20}, 
	 pat3/.style={fill=gray!20}
	 }

\tikzset
{
tree2/.pic=
{%
\draw (0,0) -- (0.5,0.5) -- (1,0);
\node[fill=black,rectangle,inner sep=2pt]  at (0,0) {};      
\node[fill=black,rectangle,inner sep=2pt]  at (1,0) {};
\node[fill=black,circle,inner sep=1pt]  at (0.5,0.5) {};
}
}

\tikzset
{
tree3/.pic=
{%
\draw (0,0) -- (0.5,0.5) -- (1,0);
\draw (0.5,0.5) -- (1,1) -- (2,0);
\node[fill=black,rectangle,inner sep=2pt]  at (0,0) {};      
\node[fill=black,rectangle,inner sep=2pt]  at (1,0) {};      
\node[fill=black,rectangle,inner sep=2pt]  at (2,0) {}; 
\node[fill=black,circle,inner sep=1pt]  at (0.5,0.5) {};
\node[fill=black,circle,inner sep=1pt]  at (1,1) {};
}
}

\tikzset
{
tree4.1/.pic=
{%
\draw (0,0) -- (0.5,0.5) -- (1,0);
\draw (0.5,0.5) -- (1,1) -- (2,0);
\draw (1,1) -- (1.5,1.5) -- (3,0);
\node[fill=black,rectangle,inner sep=2pt]  at (0,0) {};      
\node[fill=black,rectangle,inner sep=2pt]  at (1,0) {};      
\node[fill=black,rectangle,inner sep=2pt]  at (2,0) {}; 
\node[fill=black,rectangle,inner sep=2pt]  at (3,0) {}; 
\node[fill=black,circle,inner sep=1pt]  at (0.5,0.5) {};
\node[fill=black,circle,inner sep=1pt]  at (1,1) {};
\node[fill=black,circle,inner sep=1pt]  at (1.5,1.5) {};
}
}

\tikzset
{
tree4.2/.pic=
{%
\draw (0,0) -- (0.5,0.5) -- (1,0);
\draw (0.5,0.5) -- (1.5,1.5) -- (2.5,0.5);
\draw (2,0) -- (2.5,0.5) -- (3,0);
\node[fill=black,rectangle,inner sep=2pt]  at (0,0) {};      
\node[fill=black,rectangle,inner sep=2pt]  at (1,0) {};      
\node[fill=black,rectangle,inner sep=2pt]  at (2,0) {}; 
\node[fill=black,rectangle,inner sep=2pt]  at (3,0) {}; 
\node[fill=black,circle,inner sep=1pt]  at (0.5,0.5) {};
\node[fill=black,circle,inner sep=1pt]  at (1.5,1.5) {};
\node[fill=black,circle,inner sep=1pt]  at (2.5,0.5) {};
}
}

\tikzset
{
tree5.2/.pic=
{%
\draw (0,0) -- (0.5,0.5) -- (1,0);
\draw (0.5,0.5) -- (1.5,1.5) -- (2.5,0.5);
\draw (2,0) -- (2.5,0.5) -- (3,0);
\draw (1.5,1.5) -- (2,2) -- (4,0);
\node[fill=black,rectangle,inner sep=2pt]  at (0,0) {};      
\node[fill=black,rectangle,inner sep=2pt]  at (1,0) {};      
\node[fill=black,rectangle,inner sep=2pt]  at (2,0) {}; 
\node[fill=black,rectangle,inner sep=2pt]  at (3,0) {}; 
\node[fill=black,rectangle,inner sep=2pt]  at (4,0) {}; 
\node[fill=black,circle,inner sep=1pt]  at (0.5,0.5) {};
\node[fill=black,circle,inner sep=1pt]  at (2,2) {};
\node[fill=black,circle,inner sep=1pt]  at (1.5,1.5) {};
\node[fill=black,circle,inner sep=1pt]  at (2.5,0.5) {};
}
}

\title{Decks of rooted binary trees}

\author[1]{Ann Clifton\thanks{\url{aclifton@latech.edu}}}
\author[2]{\'{E}va Czabarka\thanks{\url{czabarka@math.sc.edu}}}
\author[3]{Audace Dossou-Olory\thanks{\url{audace@aims.ac.za}}}
\author[4]{Kevin Liu\thanks{\url{kliu15@uw.edu}}}
\author[5]{Sarah Loeb\thanks{\url{sloeb@hsc.edu}}}
\author[2]{Utku Okur\thanks{\url{uokur@email.sc.edu}}}
\author[2]{L\'{a}szl\'{o} Sz\'{e}kely\thanks{\url{szekely@math.sc.edu}}}
\author[6]{Kristina Wicke\thanks{\url{kristina.wicke@njit.edu}}}

\affil[1]{Louisiana Tech University, Ruston, LA, USA}
\affil[2]{University of South Carolina, Columbia, SC, USA}
\affil[3]{University of Abomey-Calavi, Abomey-Calavi, Atlantique Department, Benin}
\affil[4]{University of Washington, Seattle, WA, USA}
\affil[5]{Hampden-Sydney College, Hampden Sydney, VA, USA}
\affil[6]{New Jersey Institute of Technology, Newark, NJ, USA}

\date{}

\begin{document}
\maketitle
\begin{abstract}
\noindent We consider extremal problems related to decks and multidecks of rooted binary trees (a.k.a. rooted phylogenetic tree shapes). Here, the deck (resp. multideck) of a tree $T$ refers to the set (resp. multiset) of leaf induced binary subtrees of $T$. 
On the one hand, we consider the reconstruction of trees from their (multi)decks. We give lower and upper bounds on the minimum (multi)deck size required to uniquely encode a rooted binary tree on $n$ leaves. 
On the other hand, we consider problems related to deck cardinalities. In particular, we characterize trees with minimum-size as well as maximum-size decks. 
Finally, we present some exhaustive computations for $k$-universal trees, i.e., rooted binary trees that contain all $k$-leaf rooted binary trees as induced subtrees. 
\end{abstract}

\section{Introduction}
\noindent This paper is concerned with {\it rooted binary trees on $n$ leaves}, which correspond to {\it rooted phylogenetic tree shapes}, i.e., the unlabeled trees underlying the leaf-labeled phylogenetic trees; see \cite{SempleSteel}.
The number of  $n$-leaf rooted phylogenetic tree shapes is the Wedderburn-Etherington number $W_n$ given as sequence
A001190 in the On-Line Encyclopedia of Integer Sequences (OEIS) \cite{oeis}.
A subset of  $k$ leaves in a rooted binary tree $T$ defines a so-called {\it induced binary subtree} of a rooted binary tree,
by taking the smallest subtree of $T$ containing the $k$ leaves, designating as root in this smallest subtree
the vertex closest to the root of $T$, and suppressing non-root degree-2 vertices in the smallest subtree to obtain the induced binary subtree. The concept of induced binary subtrees is again motivated by phylogenetics, where a phylogenetic tree is created for a given set of taxa, such that the taxa correspond to the leaves of the phylogenetic tree, and the tree describes the true evolutionary history of the taxa. If a phylogenetic tree is created just for a subset of taxa, the result is exactly their induced binary tree. From a graph-theoretical viewpoint, an induced binary subtree is a special kind of topological minor of the original tree.

The {\it $k$-deck} of a rooted binary tree with $n$ leaves is the set of isomorphism classes of its $k$-leaf
induced binary subtrees, while in the {\it $k$-multideck} the isomorphism classes come with the count of how many times 
they arise. In this paper, we are concerned with extremal problems for decks and multidecks, which we now describe (note that formal definitions will be given in the next sections).

The first problem is motivated by the famous graph reconstruction conjecture that states that every graph $G$ on at least three vertices is uniquely determined (up to isomorphism) from the multiset of graphs obtained by deleting one vertex in every possible way from $G$~\cite{Kelly1942, Ulam1960}. This conjecture has been proven for many special classes of graphs but remains open in general (see, e.g.~\cite{Bondy1977, Harary1974, Ramachandran2004} for an overview). Here, we consider a restricted version of the graph reconstruction conjecture for the deck of rooted phylogenetic tree shapes. More precisely, we investigate the minimum positive integer $R(n)$ (resp. $R^{(m)}(n)$) such that every rooted binary tree on $n$ leaves is determined by its size $R(n)$-deck (resp. size $R^{(m)}(n)$-multideck). We obtain lower and upper bounds for these reconstruction numbers in Section~\ref{reconstr}. 

The second set of problems is concerned with deck cardinalities. More precisely, we analyze how small (resp. how big) the size $(n-1)$-deck of a rooted binary tree on $n$ leaves can be. We also consider the union of the size-$k$ decks for $k=1, \ldots, n$ for a rooted binary tree $T$ on $n$ leaves. Note that the cardinality of this union corresponds precisely to the number of subtrees of $T$. We characterize trees with minimum-size decks in Section~\ref{subsec: few subtrees} and trees with maximum-size decks in Section~\ref{subsec:many subtrees}. Finally, we consider the question of determining the smallest positive integer $u(k)$ such that some rooted binary tree with $u(k)$ leaves contains in its $k$-deck {\it all} rooted binary trees with $k$ leaves. 
In Section~\ref{subsec:universal}, we recall asymptotic upper and lower bounds for $u(k)$ from the literature. Additionally, we explicitly compute $u(k)$ and enumerate the trees realizing $u(k)$, so-called {\it $k$-universal trees}, for small $k$.

We conclude by stating some open problems and directions for future research in Section~\ref{sec:openproblems}.

\section{Preliminaries}
\label{prelim}
\begin{definition}
For any set $A$ and nonnegative integer $j$, $\binom{A}{j}=\left\{X\subseteq A: |X|=j\right\}$. Let $[n] = \{1,\ldots,n\}$.
\end{definition}

\begin{definition}
A \emph{rooted tree} is a tree $T$ where a special vertex, denoted by $r_T$, is identified as the root of $T$. The \emph{parent} of a non-root vertex $v$ is $v$'s neighbor on the $r_T$-$v$ path in $T$ and $y$ is a \emph{child of $v$} if $v$ is the parent of $y$. The \emph{leaves} of $T$ are the vertices with no children. We let $\bbL(T)$ denote the set of leaves and $|T|$ denote the number of leaves. If $|T|=n$ we say that $T$ is a \emph{size-$n$ tree}.
\end{definition}

We remark that whenever we write $T = T'$,  $T$ and $T'$ are equal up to isomorphism. 

\begin{definition}
A \emph{rooted binary tree} is a rooted tree $T$ where every vertex has either no children (i.e., is a leaf) or two children.
\end{definition}

Note that in a rooted binary tree we have that either the root is a leaf (and the size of the tree is $1$) or the root has degree $2$ and non-root internal vertices have degree $3$. One can check that for $n\le 3$ the size-$n$ rooted binary trees are unique.

\begin{definition} Let $T_1,T_2$ be rooted binary trees. Let $T_1\oplus T_2$ be the rooted binary tree $T$ obtained by taking a new vertex as $r_T$ and joining the roots of $T_1$ and $T_2$ to $r_T$. 
If $T=T_1\oplus T_2$, we say $T$ is composed of $T_1$ and $T_2$ and the root-split of $T$ is the set 
$\{|T_1|,|T_2|\}$ 
(which has only one element when $|T_1|=|T_2|$). 
\end{definition}

An example of this operation is shown in Figure~\ref{fig:composed}. Observe that the root split of $T_1\oplus T_2$ is $\{4,5\}$. Since we will consider unlabeled rooted binary trees up to isomorphism, we consider this operation to be commutative. 

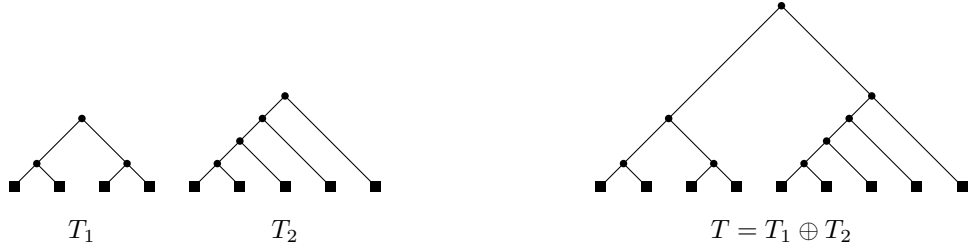
\begin{figure}[htbp]
\centering
\begin{tikzpicture}[scale=.6]

    \node[fill=black,rectangle,inner sep=2pt]   at (7,0) {};
  \node[fill=black,rectangle,inner sep=2pt]   at (8,0) {};
   \node[fill=black,rectangle,inner sep=2pt]   at (9,0) {};      
   \node[fill=black,rectangle,inner sep=2pt]   at (10,0) {};      
         \node[fill=black,circle,inner sep=1pt]  at (7.5,0.5) {};
         \node[fill=black,circle,inner sep=1pt]  at (9.5,.5) {};
            \node[fill=black,circle,inner sep=1pt]  at (8.5,1.5) {};
         \draw (7,0)--(7.5,.5)--(8,0);
          \draw (7.5,.5)--(8.5,1.5)--(9.5,0.5);
         \draw (9,0)--(9.5,.5)--(10,0);
         \node at (8.5,-1) {$T_1$};

      \node[fill=black,rectangle,inner sep=2pt]   at (11,0) {};
  \node[fill=black,rectangle,inner sep=2pt]   at (12,0) {};
   \node[fill=black,rectangle,inner sep=2pt]   at (13,0) {};      
   \node[fill=black,rectangle,inner sep=2pt]   at (14,0) {};     
   \node[fill=black,rectangle,inner sep=2pt]   at (15,0) {};      
         \node[fill=black,circle,inner sep=1pt]  at (11.5,0.5) {};
         \node[fill=black,circle,inner sep=1pt]  at (12,1) {};
            \node[fill=black,circle,inner sep=1pt]  at (12.5,1.5) {};
            \node[fill=black,circle,inner sep=1pt]  at (13,2) {};   
         \draw (11,0)--(11.5,.5)--(12,0);
          \draw (11.5,.5)--(12,1)--(13,0);
         \draw (12,1)--(12.5,1.5)--(14,0);
           \draw (12.5,1.5)--(13,2)--(15,0);
         \node at (13,-1) {$T_2$};

      \node[fill=black,rectangle,inner sep=2pt]   at (20,0) {};
  \node[fill=black,rectangle,inner sep=2pt]   at (21,0) {};
   \node[fill=black,rectangle,inner sep=2pt]   at (22,0) {};      
   \node[fill=black,rectangle,inner sep=2pt]   at (23,0) {};      
         \node[fill=black,circle,inner sep=1pt]  at (20.5,0.5) {};
         \node[fill=black,circle,inner sep=1pt]  at (22.5,.5) {};
            \node[fill=black,circle,inner sep=1pt]  at (21.5,1.5) {};
         \draw (20,0)--(20.5,.5)--(21,0);
          \draw (20.5,.5)--(21.5,1.5)--(22.5,0.5);
         \draw (22,0)--(22.5,.5)--(23,0);

      \node[fill=black,rectangle,inner sep=2pt]   at (24,0) {};
  \node[fill=black,rectangle,inner sep=2pt]   at (25,0) {};
   \node[fill=black,rectangle,inner sep=2pt]   at (26,0) {};      
   \node[fill=black,rectangle,inner sep=2pt]   at (27,0) {};     
   \node[fill=black,rectangle,inner sep=2pt]   at (28,0) {};      
         \node[fill=black,circle,inner sep=1pt]  at (24.5,0.5) {};
         \node[fill=black,circle,inner sep=1pt]  at (25,1) {};
            \node[fill=black,circle,inner sep=1pt]  at (25.5,1.5) {};
            \node[fill=black,circle,inner sep=1pt]  at (26,2) {};   
         \draw (24,0)--(24.5,.5)--(25,0);
          \draw (24.5,.5)--(25,1)--(26,0);
         \draw (25,1)--(25.5,1.5)--(27,0);
           \draw (25.5,1.5)--(26,2)--(28,0);
            \node[fill=black,circle,inner sep=1pt]  at (24,4) {};
            \draw (21.5,1.5)--(24,4)--(26,2);
            \node at (24,-1) {$T=T_1\oplus T_2$};
\end{tikzpicture} 
\caption{The tree $T$ is composed of $T_1$ and $T_2$.}
\label{fig:composed}
\end{figure}

\begin{definition}
Let $i\in\mathbb{N}$ and $T_1,T_2$ be rooted binary trees. We define the operations $(T_1\oplus)^i  T_2$  as follows:
$(T_1\oplus)^0T_2=T_2$ and for $i>0$, $(T_1\oplus)^i(T_2)=T_1\oplus((T_1\oplus)^{i-1}T_2)$.
\end{definition}

\begin{definition} The \emph{size-$n$ caterpillar $C_n$} is the rooted binary tree $C_n=(C_1\oplus)^{n-1} C_{1}$, where $C_1$ is the singleton vertex rooted tree.
\end{definition}

\begin{definition} The \emph{height-$h$ complete binary tree $B_h$} is defined as follows: $B_0=C_1$ and for $h>0$
$B_h=B_{h-1}\oplus B_{h-1}$
\end{definition}

\begin{definition}
Let $T$ be a rooted binary tree. For a set $S \subseteq \mathbb{L}(T)$, the \emph{rooted binary subtree $T[S]$ induced by $S$} is obtained by first taking the minimal connected subgraph $T'$ of $T$ containing $S$, letting $r_{T[S]}$ be the closest vertex of $T'$ to $r_T$, and then suppressing all other degree-2 vertices in $T'$.
\end{definition}

For example, consider the tree $T$ with $|T|=5$ shown on the left in Figure~\ref{fig:inducedsubtree}. Four choices of $S\subseteq \mathbb{L}(T)$ induce $C_4$ and one choice induces $B_2$, which are shown on the right.

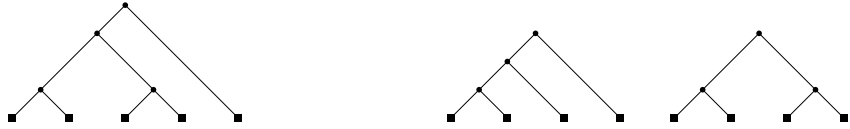
\begin{figure}
\centering
\scalebox{0.75}{
\begin{tikzpicture}
\pic at (0,0) {tree5.2};
\end{tikzpicture}
\qquad 
\qquad \qquad \qquad \qquad
\begin{tikzpicture}
\pic at (0,0) {tree4.1};
\end{tikzpicture}
\qquad 
\begin{tikzpicture}
\pic at (0,0) {tree4.2};
\end{tikzpicture}
}
\caption{For the rooted binary tree on the left, the possible size-4 induced subtrees are shown on the right.}
\label{fig:inducedsubtree}
\end{figure}

\begin{definition} Let $A$ be a multiset. For any object $x$, $m_A(x)$ denotes the multiplicity of $x$ in $A$ (in particular, $m_A(x)>0$ means that $x\in A$).
\end{definition}

\begin{definition}
Let $T$ be a size-$n$ rooted binary tree and $j\in[n]$. 
The \emph{size-$j$ deck $\D_j(T)$ of $T$} is 
\[\D_j(T)=\left\{T[S]: S\in\binom{\bbL(T)}{j}\right\}.\]
The \emph{size-$j$ multideck $\Dm_j(T)$ of $T$} is the multiset of size-$j$ trees, where for any size-$j$ tree $T'$ 
\[m_{\Dm_j(T)}(T')=
\left\vert\left\{S\in\binom{\bbL(T)}{j}:T[S]= T'\right\}\right\vert.\]
\end{definition}

Note that if $T$ is of size $n$, then $\D_n(T)=\Dm_n(T)=\{T\}$. Moreover, note that when we consider the size-$(\vert T \vert -1)$ (multi)deck of a tree, we often use the notation $\mathcal{D}(T)=\mathcal{D}_{|T|-1}(T)$ and $\mathcal{D}^{(m)}(T)=\mathcal{D}^{(m)}_{|T|-1}(T)$, and we refer to them as the deck and multideck of $T$.

\begin{definition}
Let $n$ be a positive integer and $j\in[n]$. We say that the \emph{size-$j$ decks} (resp. \emph{multidecks}) \emph{determine the size-$n$ trees} if
whenever $T_1,T_2$ are size-$n$ trees with $T_1 \neq T_2$,
we have $\D_j(T_1)\ne\D_j(T_2)$ (resp. $\Dm_j(T_1)\ne\Dm_j(T_2)$). If $T$ is a size-$n$ tree, 
the \emph{size-$j$ deck} (resp. \emph{multideck})  \emph{determines $T$} if
whenever $T_1$ is size-$n$ tree with $T_1 \neq T$,
we have $\D_j(T_1)\ne\D_j(T)$ (resp. $\Dm_j(T_1)\ne\Dm_j(T)$).
\end{definition}

It is obvious that the size-$n$ decks and multidecks determine the size-$n$ trees.

\begin{lemma} Let $i,j,n$ be positive integers with $i<j\le n$.
If the size-$i$ decks (resp. multidecks) determine the size-$n$ trees, then so do the size-$j$ decks (resp. multidecks).
\end{lemma}

\begin{proof}
It is enough to show that the size-$j$ decks (resp. multidecks) determine the size-$i$ decks (resp. multidecks). This follows for the decks easily, as for $S\subseteq S'\subseteq\bbL(T)$ we have $(T[S'])[S]=T[S]$, consequently
\[\D_i(T)=\bigcup_{T'\in\D_j(T)} \D_i(T').\]

Consider the corresponding statement for multidecks. Let
$A,B$ be multisets of trees, and $A\cup B$ be the multiset of trees defined by the following: for every tree $T'$, $m_{A\cup B}(T')=m_A(T')+m_B(T')$. Consider the multiset 
\[\mathcal{X}=\bigcup_{T'\in\Dm_j(T)} \Dm_i(T').\]
By the previous observation, $\mathcal{X}$ has the same elements as $\D_i(T)$, with not necessarily the same multiplicities as they appear in $\Dm_i(T)$. We claim
that for any $T'\in\D_i(T)$, $m_{\mathcal{X}}(T')=\binom{n-i}{j-i}m_{\Dm_i(T)}(T')$. This will
prove the claim, as we can obtain $\Dm_i(T)$ from $\mathcal{X}$.

Let $T'\in\D_i(T)$ and let
$M=m_{\Dm_i(T)}(T')$. Let $S_1,\ldots,S_M$
be an enumeration of the size-$i$ leaf sets $S$ with $T[S]= T'$. Each $S_i$ contributes $1$ to the multiplicity
of $T'$ in $\Dm_i(T)$, and $\binom{n-i}{j-i}$ to the multiplicity
of $T'$ in $\mathcal{X}$, as there are $\binom{n-i}{j-i}$ size-$j$ leaf sets of $T$ that contain $S_i$. 
\end{proof}

\section{Reconstruction} \label{reconstr}
In this section, we consider the reconstruction problem for decks and multidecks of rooted binary trees, which we now formally define.

\begin{definition}
For each positive integer $n \geq 4$, let
$R(n)$ (resp. $R^{(m)}(n)$) denote the minimum positive integer $j\in[n]$ such that the size $R(n)$-deck (resp. size $R^{(m)}(n)$-multideck) determines the size-$n$ trees.  We refer to $R(n)$ and $R^{(m)}(n)$ as the deck reconstruction number and multi-deck reconstruction number, respectively. 
\end{definition}

Note that $R(n)$ (resp. $R^{(m)}(n)$) is only defined for $n \geq 4$, since for $n \leq 3$, the size-$n$ binary trees are unique and therefore trivially reconstructible. 

For the same reason, i.e., since for $k \le 3$, the size-$k$ binary trees are unique, for $n \ge 4$ we have $R(n) \ge 4$ and $R^{(m)}(n) \ge 4$. We noted earlier that $R(n) \le n$ and $R^{(m)}(n) \le n$. Consequently, $R(4)=R^{(m)}(4)=4$. Figure~\ref{fig:smalltrees} shows that $R^{(m)}(5) = 4$ and $R(5) > 4$, which implies  $R(5) = 5$. 

\begin{figure}[!ht]
\centering
\begin{tikzpicture}[scale=.6]
\node[fill=black,rectangle,inner sep=2pt]   at (-3,0) {};
\node at (-3,.5) {$C_1=B_0$};

   \node[fill=black,rectangle,inner sep=2pt]   at (-1,0) {};
   \node[fill=black,rectangle,inner sep=2pt]   at (0,0) {};
          \node[fill=black,circle,inner sep=1pt]  at (-.5,0.5) {};
          \draw (-1,0)--(-.5,.5)--(0,0);
\node at (-.5,1) {$C_2=B_1$};

  \node[fill=black,rectangle,inner sep=2pt]   at (2,0) {};
   \node[fill=black,rectangle,inner sep=2pt]   at (3,0) {};
    \node[fill=black,rectangle,inner sep=2pt]   at (4,0) {};      
          \node[fill=black,circle,inner sep=1pt]  at (2.5,0.5) {};
          \node[fill=black,circle,inner sep=1pt]  at (3,1) {};
          \draw (2,0)--(2.5,.5)--(3,0);
           \draw (2.5,.5)--(3,1)--(4,0);
\node at (3,1.5) {$C_3$};

     \node[fill=black,rectangle,inner sep=2pt]   at (6,0) {};
   \node[fill=black,rectangle,inner sep=2pt]   at (7,0) {};
    \node[fill=black,rectangle,inner sep=2pt]   at (8,0) {};      
    \node[fill=black,rectangle,inner sep=2pt]   at (9,0) {};      
          \node[fill=black,circle,inner sep=1pt]  at (6.5,0.5) {};
          \node[fill=black,circle,inner sep=1pt]  at (7,1) {};
             \node[fill=black,circle,inner sep=1pt]  at (7.5,1.5) {};
          \draw (6,0)--(6.5,.5)--(7,0);
           \draw (6.5,.5)--(7,1)--(8,0);
          \draw (7,1)--(7.5,1.5)--(9,0);
\node at (7.5,2) {$C_4$};

     \node[fill=black,rectangle,inner sep=2pt]   at (11,0) {};
   \node[fill=black,rectangle,inner sep=2pt]   at (12,0) {};
    \node[fill=black,rectangle,inner sep=2pt]   at (13,0) {};      
    \node[fill=black,rectangle,inner sep=2pt]   at (14,0) {};      
          \node[fill=black,circle,inner sep=1pt]  at (11.5,0.5) {};
          \node[fill=black,circle,inner sep=1pt]  at (13.5,.5) {};
             \node[fill=black,circle,inner sep=1pt]  at (12.5,1.5) {};
          \draw (11,0)--(11.5,.5)--(12,0);
           \draw (11.5,.5)--(12.5,1.5)--(13.5,0.5);
          \draw (13,0)--(13.5,.5)--(14,0);
\node at (12.5,2) {$B_2$};              
\end{tikzpicture}
\begin{tikzpicture}[scale=.6]
\node[] at (14,3) {}; 
       \node[fill=black,rectangle,inner sep=2pt]   at (11,0) {};
   \node[fill=black,rectangle,inner sep=2pt]   at (12,0) {};
    \node[fill=black,rectangle,inner sep=2pt]   at (13,0) {};      
    \node[fill=black,rectangle,inner sep=2pt]   at (14,0) {};     
    \node[fill=black,rectangle,inner sep=2pt]   at (15,0) {};      
          \node[fill=black,circle,inner sep=1pt]  at (11.5,0.5) {};
          \node[fill=black,circle,inner sep=1pt]  at (12,1) {};
             \node[fill=black,circle,inner sep=1pt]  at (12.5,1.5) {};
             \node[fill=black,circle,inner sep=1pt]  at (13,2) {};   
          \draw (11,0)--(11.5,.5)--(12,0);
           \draw (11.5,.5)--(12,1)--(13,0);
          \draw (12,1)--(12.5,1.5)--(14,0);
            \draw (12.5,1.5)--(13,2)--(15,0);
\node at (13,-1) {$(5,0)$};

      \node[fill=black,rectangle,inner sep=2pt]   at (19,0) {};
   \node[fill=black,rectangle,inner sep=2pt]   at (20,0) {};
    \node[fill=black,rectangle,inner sep=2pt]   at (21,0) {};      
    \node[fill=black,rectangle,inner sep=2pt]   at (22,0) {};    
    \node[fill=black,rectangle,inner sep=2pt]   at (23,0) {};           
          \node[fill=black,circle,inner sep=1pt]  at (19.5,0.5) {};
          \node[fill=black,circle,inner sep=1pt]  at (21.5,.5) {};
             \node[fill=black,circle,inner sep=1pt]  at (20.5,1.5) {};
              \node[fill=black,circle,inner sep=1pt]  at (21,2) {};            
          \draw (19,0)--(19.5,.5)--(20,0);
           \draw (19.5,.5)--(20.5,1.5)--(21.5,0.5);
          \draw (21,0)--(21.5,.5)--(22,0);
         \draw (20.5,1.5)--(21,2)--(23,0);
\node at (21,-1) {$(4,1)$};

     \node[fill=black,rectangle,inner sep=2pt]   at (30,0) {};
   \node[fill=black,rectangle,inner sep=2pt]   at (31,0) {};
          \node[fill=black,circle,inner sep=1pt]  at (30.5,0.5) {};
          \draw (30,0)--(30.5,.5)--(31,0);            
  \node[fill=black,rectangle,inner sep=2pt]   at (27,0) {};
   \node[fill=black,rectangle,inner sep=2pt]   at (28,0) {};
    \node[fill=black,rectangle,inner sep=2pt]   at (29,0) {};      
          \node[fill=black,circle,inner sep=1pt]  at (27.5,0.5) {};
          \node[fill=black,circle,inner sep=1pt]  at (28,1) {};
             \node[fill=black,circle,inner sep=1pt]  at (29,2) {};  
          \draw (27,0)--(27.5,.5)--(28,0);
           \draw (27.5,.5)--(28,1)--(29,0);
       \draw (28,1)--(29,2)--(30.5,0.5);
\node at (29,-1) {$(2,3)$};   
\end{tikzpicture} 
\caption{Binary trees of size at most $5$. The pairs $(i,j)$ under the size-$5$ trees show the multiplicities of $C_4$ and $B_2$
in the size-$4$ multideck of the tree.}
\end{figure}
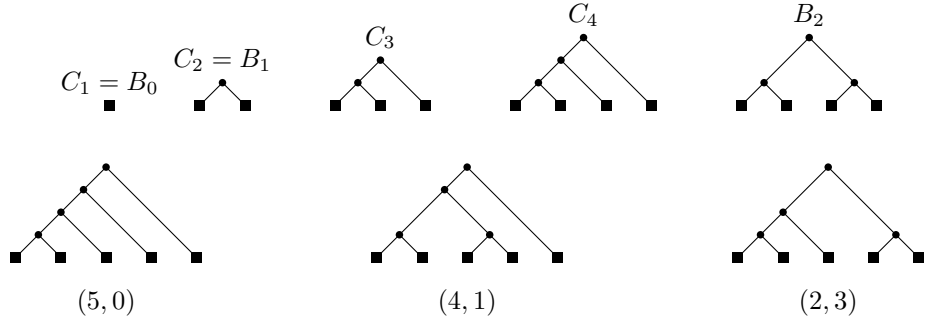\label{fig:smalltrees}

\subsection{Lower bound}

We begin by establishing a lower bound for the deck size required for reconstruction.
\begin{theorem}\label{decklowerbound}
For any $n\geq 6$, we have $R(n)> 2\lceil\frac{n}{4}\rceil$. 
\end{theorem}

This will follow immediately from
Lemmas~\ref{lm:0mod4},~\ref{lm:3mod4},~\ref{lm:2mod4}, and~\ref{lm:1mod4}, which we will state and prove in this section. In these lemmas, we make extensive use of trees formed by starting with $B_2$ and identifying leaves with caterpillars. Such trees are of the form $(C_i \oplus C_j) \oplus (C_k \oplus C_\ell)$. Examples of $(C_3\oplus C_1)\oplus (C_2\oplus C_2)$, $(C_3\oplus C_2)\oplus (C_2\oplus C_1)$, and $(C_3\oplus C_1)\oplus (C_2\oplus C_1)$, as in the construction of Lemma~\ref{lm:0mod4} for $k=2$, appear in Figure~\ref{fig:0mod4}. Throughout, we use the fact that every induced subtree of a caterpillar $C_n$ is another caterpillar $C_k$ with $1\leq k\leq n$. 

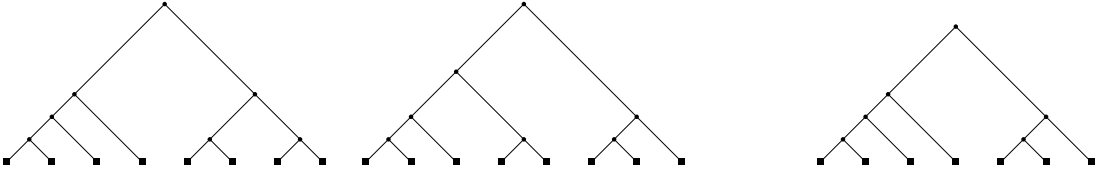
\begin{figure}
\centering
\scalebox{0.6}{
\begin{tikzpicture}
\pic at (0,0) {tree4.1};
\pic at (4,0) {tree2};
\pic at (6,0) {tree2};
\draw (1.5,1.5) -- (3.5,3.5) -- (6.5,0.5);
\draw (4.5,0.5) -- (5.5,1.5);
\node[fill=black,circle,inner sep=1pt]  at (5.5,1.5) {};
\node[fill=black,circle,inner sep=1pt]  at (3.5,3.5) {};
\end{tikzpicture}
\qquad 
\begin{tikzpicture}
\pic at (0,0) {tree3};
\pic at (3,0) {tree2};
\pic at (5,0) {tree3};
\draw (1,1) -- (2,2) -- (3.5,0.5);
\draw (2,2) -- (3.5,3.5) -- (6,1);
\node[fill=black,circle,inner sep=1pt]  at (3.5,3.5) {};
\node[fill=black,circle,inner sep=1pt]  at (2,2) {};
\end{tikzpicture}
\qquad 
\qquad 
\qquad 
\qquad
\begin{tikzpicture}
\pic at (0,0) {tree4.1};
\pic at (4,0) {tree3};
\draw (1.5,1.5) -- (3,3) -- (5,1);
\node[fill=black,circle,inner sep=1pt]  at (3,3) {};
\end{tikzpicture}
}
\caption{The two 8-leaf trees on the left and the 7-leaf tree on the right have the same size-4 deck.}
\label{fig:0mod4}
\end{figure}

\begin{lemma}\label{lm:0mod4}
For any $k\ge 2$ there exist nonisomorphic trees $T_1$ and $T_2$ of size $4k$ such that $\mathcal{D}_{2k}(T_1)=\mathcal{D}_{2k}(T_2)$. 
\end{lemma}

\begin{proof}
Let $k\ge 2$ and define 
\begin{align*}
T_1&=(C_{k+1}\oplus C_{k-1})\oplus (C_k\oplus C_k), \\ 
T_2&=(C_{k+1}\oplus C_{k})\oplus (C_{k}\oplus C_{k-1}), \text{ and } \\ 
S&=(C_{k+1}\oplus C_{k-1})\oplus (C_k\oplus C_{k-1})
\end{align*}
One can check that $T_1 \neq T_2$. We claim that $\mathcal{D}_{2k}(T_1)=\D_{2k}(S)=\D_{2k}(T_2)$. Since $S$ is an induced subtree of $T_i$ for each $i$, we have $\mathcal{D}_{2k}(S)\subseteq \mathcal{D}_{2k}(T_i)$. It remains to show the other containment for each tree $T_i$.

Assume $T\in \mathcal{D}_{2k}(T_1)$. Then it is possible for $T=C_{k+1}\oplus C_{k-1}$, or $T=C_k\oplus C_k$, in which case $T \in D_{2k}(S)$. Otherwise, $T=R_1\oplus R_2$ where $R_1\in\D_{2k-j}(C_{k+1}\oplus C_{k-1})$ and 
$R_2\in \D_{j}(C_k\oplus C_k)$ for some $j\in[2k-1]$. Since then $R_2 \in \D_j( C_k \oplus C_{k-1})$ when $j \in [2k-1]$, we have $T \in D_{2k}(S)$.

Assume now that $T \in \D_{2k}(T_2)$. Then it is possible for $T =C_k \oplus C_k$, or $T=C_{k+1} \oplus C_{k-1}$, in which case $T \in D_{2k}(S)$. Otherwise, $T=R_1 \oplus R_2$, where $R_1 \in \D_{j}(C_{k+1} \oplus C_k)$ and $R_2 \in \D_{2k-j}(C_k \oplus C_{k-1})$ for some $j \in [2k-1]$. Since then $R_1 \in \D_j(C_{k+1} \oplus C_{k-1})$ when $j \in [2k-1]$, we have $T \in D_{2k}(S)$.
\end{proof}

\begin{lemma}\label{lm:3mod4}
For any $k\ge 3$ there exist nonisomorphic trees $T_1$ and $T_2$
of size $4k-1$ such that 
$\mathcal{D}_{2k}(T_1)=\mathcal{D}_{2k}(T_2)$.
\end{lemma}

\begin{proof}
Let $k \ge 3$ and define
\begin{align*}
T_1 & =(C_{k+1}\oplus C_k)\oplus (C_k\oplus C_{k-2}), \\
T_2 & =(C_{k+1}\oplus C_{k-1})\oplus (C_k\oplus C_{k-1}), \text{ and } \\
S & =(C_{k+1}\oplus C_{k-1})\oplus (C_k\oplus C_{k-2}).
\end{align*}
One can see that $T_1 \neq T_2$. We claim that for $\mathcal{D}_{2k}(T_1)=\mathcal{D}_{2k}(S) = \mathcal{D}_{2k}(T_2)$.
Since $S$ is a subtree of $T_i$ for each $i$, we have 
$\D_{2k}(S) \subseteq \D_{2k}(T_i)$. 

Let $T\in\D_{2k}(T_1)$. If
$T$ is one of $C_{k+1}\oplus C_{k-1},C_k\oplus C_k,$ or $
C_2\oplus(C_k\oplus C_{k-2})$, then $T \in \D_{2k}(S)$. Otherwise, $T=R_1\oplus R_2$ for
some $R_1\in\D_{2k-j}(C_{k+1}\oplus C_k)$ and $R_2\in\D_{j}(C_k\oplus C_{k-2})$ for some $j\in[2k-3]$. As in the proof of Lemma~\ref{lm:0mod4}, we have $R_1 \in \D_{2k-j}(C_{k+1}\oplus C_{k-1})$ when $j \in [2k-3]$, so $T \in \mathcal{D}_{2k}(S)$.

Let $T\in\D_{2k}(T_2)$. If $T$ is either of
$C_{k+1}\oplus C_{k-1}$ or $C_1\oplus(C_k\oplus C_{k-1}) = (C_k \oplus C_{k-1}) \oplus C_1$, then $T \in \D_{2k}(S)$. Otherwise, $T=R_1\oplus R_2$ for
some $R_1\in\D_{2k-j}(C_{k+1}\oplus C_{k-1})$ and 
$R_2\in\D_{j}(C_k\oplus C_{k-1})$ for some $j\in[2k-2]$. If $R_2 = (C_{k-1} \oplus C_{k-1})$, then $R_1 = C_2$, and $T = (C_{k-1} \oplus C_{k-1}) \oplus C_2 \in D_{2k}(S)$. Otherwise, $R_2 \in \mathcal{D}_j(C_k \oplus C_{k-2})$ for some $j \in [2k-2]$ and $T \in \mathcal{D}_{2k}(S)$. 
\end{proof}

\begin{lemma}\label{lm:2mod4}
For any $k\ge 2$ there exist nonisomorphic trees $T_1$ and $T_2$ of
size $4k-2$ such that $\mathcal{D}_{2k}(T_1)=\mathcal{D}_{2k}(T_2)$. 
\end{lemma}

\begin{proof}
Let $k\ge 2$ and define 
\begin{align*}
T_1 & =(C_k\oplus C_{k-1})\oplus (C_k \oplus C_{k-1}), \\
T_2 & =(C_k\oplus C_k)\oplus (C_{k-1}\oplus C_{k-1}), \text{ and } \\
S & =(C_{k}\oplus C_{k-1})\oplus (C_{k-1}\oplus C_{k-1}).
\end{align*}
One can see $T_1 \neq T_2$. We claim that $\mathcal{D}_{2k}(T_1)=\mathcal{D}_{2k}(S)=\mathcal{D}_{2k}(T_2).$ Since $S$ is a subtree of $T_i$ for each $i$, we have $\mathcal{D}_{2k}(S)\subseteq \mathcal{D}_{2k}(T_i)$.

For the other containment, consider $T\in \mathcal{D}_{2k}(T_1)\cup\D_{2k}(T_2)$.
Then we can have $T = C_k\oplus C_k = C_k \oplus (C_1 \oplus C_{k-1}) \in \mathcal{D}_{2k}(S)$. Otherwise, one of the following holds: $T = C_k \oplus R$ for some $R\in \D_k(C_{k-1}\oplus C_{k-1})$, $T = R_1 \oplus R_2$ with $R_i \in \D_k(C_{k-1}\oplus C_{k-1})$ for each $i$, or $T = Q_1 \oplus Q_2$ where $Q_1\in\D_j(C_{k-1}\oplus C_{k-1})$ and $Q_2\in\D_{2k-j} (C_k\oplus C_{k-1})$ for some $j \in [k-1]$. In each case, $T\in\D_{2k}(S)$.
\end{proof}

\begin{lemma}\label{lm:1mod4}
For $k\ge 2$ there exist nonisomorphic trees $T_1$ and $T_2$ of size $4k-3$
such that 
$\mathcal{D}_{2k}(T_1)=\mathcal{D}_{2k}(T_2)$.
\end{lemma}

\begin{proof}
If $k=2$, let $T_1=(C_2\oplus C_2)\oplus C_1$ and $T_1=C_3\oplus C_2$. 
Then $T_1 \neq T_2$, but
$\D_4(T_1)=\D_4(T_2)=\{C_2\oplus C_2,C_4\}$.

Let $k\ge 3$ and define 
\begin{align*}
T_1 & =(C_{k}\oplus C_{k})\oplus (C_{k-1}\oplus C_{k-2}) , \\ 
T_2 & =(C_{k}\oplus C_{k-1})\oplus (C_{k}\oplus C_{k-2}) , \text{ and } \\
S & =(C_{k}\oplus C_{k-1})\oplus (C_{k-1}\oplus C_{k-2})
\end{align*}
One can check $T_1 \neq T_2$. Since $S$ is a subtree of each $T_i$, $\D_{2k}(S)\subseteq \D_{2k}(T_i)$ for each $i$. 
We will show that 
$\D_{2k}(T_1)=\D_{2k}(S)= \D_{2k}(T_2).$

Let $T\in\D_{2k}(T_1)$. We can have $T = C_k \oplus C_k = C_k \oplus (C_{1} \oplus C_{k-1}) \in \D_{2k}(S)$. Otherwise $T = R_1 \oplus R_2$ for some $R_1 \in\D_{2k-j}(C_{k}\oplus C_{k})$ and $R_2 \in  \D_j(C_{k-1}\oplus C_{k-2})$ for some $j \in [2k-3]$. Since $\D_{2k-j}(C_{k}\oplus C_{k})=\D_{2k-j}(C_k\oplus C_{k-1})$ for $j \in [2k-3]$, we have $T \in \D_{2k}(S)$. 

Let $T \in \D_{2k}(T_2)$. Then $T = R_1 \oplus R_2$ where $R_1 \in D_{2k-j}(C_k \oplus C_{k-1})$ and $R_{2} \in D_{j}(C_{k} \oplus C_{k-2})$ for some $j \in [2k-2]$. If $1 \le j \le k$, then $R_2 \in \D_{j}(C_{k-1} \oplus C_{k-2})$, so $T \in \D_{2k}(S)$. If instead $k+1 \le j \le 2k-2$, we have $2k-j\le k-1$ and so $R_1 \in \D_{2k-j}(C_{k-1} \oplus C_{k-2})$ and $R_2 \in D_{j}(C_k \oplus C_{k-1})$. Hence $T \in D_{2k}(S)$. 
\end{proof}

We turn to finding a lower bound on the multideck size required for reconstruction. The number of rooted $n$-leaf binary trees is the Wedderburn-Etherington number $W_n$
described in~\cite[A001190]{oeis}. By~\cite{number_of_trees}, $W_n \sim \alpha \cdot \frac{b^n}{n^{3/2}}$ for some positive constants $\alpha, b$. In particular, 
\begin{equation} \label{eqn:Wnboundsb}
    C_1 \cdot \frac{b^n}{n^{3/2}}\leq W_n \leq C_2 \cdot \frac{b^n}{n^{3/2}}
\end{equation} for some positive constants $C_1$ and $C_2$, for every $n\geq 1$. This formula will be recalled throughout the paper, particularly in Section~\ref{subsec:many subtrees}. Numerically, $b \approx 2.4832$. The following lemma compares $W_i$ to $\binom{n}{i}$. 

\begin{lemma} \label{lem:Wkbounds}
If $n$ is sufficiently large, then there exists a positive integer $j_0$, such that for all $1\leq i\leq j_0$,
 $W_i\leq \binom{n}{i} $ and for all $j_0\leq i\leq n$, $W_i\geq \binom{n}{i} $, and furthermore,
 $ 0.685n\leq j_0\leq 0.6851n$. 
\end{lemma}

\begin{proof}
Let $C_1, C_2$, and $b$ be as in Equation~\eqref{eqn:Wnboundsb}. 

First we consider the case when $1\leq i \le n/2$. We  find $i(1+\frac{b}{e})<2i\leq n$, so $\frac{bi}{e}<n-i$. Thus $b^i\bigl(\frac{i}{e}\bigl)^i<(n-i)^i$. 
By Stirling's formula, there exists a number $i_0$, independent of $n$, such that for all $i \ge i_0$, we have $C_2 i! < i^{3/2} (\frac{i}{e})^i$.
Hence, for all $i$ with $i_0\leq i\leq \frac{n}{2}$  and  $n$ sufficiently large, \[W_i\leq C_2\frac{b^i}{i^{3/2}} < C_2\frac{(n-i)^i}{i^{3/2} (i/e)^i} < \frac{(n-i)^i}{i!} \leq \binom{n}{i}.\]  
On the other hand, since the sequence of $W_i$ is increasing, for $1 \le i \le i_0$, we can bound $W_i$ from above by the constant $W_{i_0}$. 
We conclude that for every sufficiently large $n$,  $W_i < \binom{n}{i}$ for all $1\leq i \le n/2$.

On the interval $\frac{n}{2}\leq i \leq n$, we have that the sequence $\binom{n}{i}$ is decreasing and the sequence $W_i$ is increasing. Therefore the inequality between them switches exactly once in the range $n/2 \le i \le n$. 
Recall from e.g. formulas (11.42) and (11.50) in~\cite{Cover-Thomas} that for any real number $0<c<1$, such that $cn$ is an integer, the following explicit bounds hold:
\[ \frac{1}{n+1} \left(c^{-c}(1-c)^{c-1}\right)^n \leq    \binom{n}{cn}\leq \left(c^{-c}(1-c)^{c-1}\right)^n.\]

One can check computationally, e.g. with MAPLE, that for $0.5<c\leq 0.685$, we have $b^c < c^{-c}(1-c)^{c-1}$. Thus, for $i = cn$ with $0.5n \le i  \le 0.685n$, there is an $\epsilon > 0$ for which 
\[\binom{n}{i}\ge \frac{1}{n+1}\left[ \left(\frac{i}{n}\right)^{-\frac{i}{n}} \left(1-\frac{i}{n}\right)^{\frac{i}{n}-1} \right]^n \ge \frac{1}{n+1} \left((b+\epsilon)^\frac{i}{n}\right)^n\ge \frac{1}{2i+1}(b+\epsilon)^i \ge C_2 \frac{b^i}{i^{3/2}} \ge W_i\]
for $n$ sufficiently large.

Similarly one can check computationally that for $c\geq 0.6851$, we have $b^c > c^{-c}(1-c)^{c-1}$. Thus, for $i = cn$ with $i \ge 0.6851n$, there exists an $\epsilon >0$ such that 
\[\binom{n}{i}\le  \left[ \left(\frac{i}{n}\right)^{-\frac{i}{n}} \left(1-\frac{i}{n}\right)^{\frac{i}{n}-1}  \right]^n \le \left((b-\epsilon)^\frac{i}{n}\right)^n=(b-\epsilon)^i\le  C_1 \frac{b^i}{i^{3/2}} \le W_i\]
for $n$ sufficiently large.
\end{proof}

We use the $\log$ notation for the natural logarithm.  
\begin{theorem} \label{thm:multilow}
We have $R^{(m)}(n)\geq (1-o(1))\frac{\log n}{\log b}$ with $b$ as in Equation~\eqref{eqn:Wnboundsb}.
\end{theorem}

\begin{proof}
Let $\mathcal{T}_n$ denote the set of $n$-leaf rooted binary trees. Assume that $R^{(m)}(n)=k$.
For reconstruction from multidecks $\{\mathcal{D}_k^m(T)\}_{T\in \mathcal{T}_n}$, there must be at least as many compositions of the integer $\binom{n}{k}$ into $W_k$ terms as there are members of $\mathcal{T}_n$. Since $|\mathcal{T}_n| = W_n$, using the formula for the number of compositions, the inequality \begin{equation*}\binom{\binom{n}{k} +W_k -1}{W_k-1}\geq W_n \end{equation*} is necessary for reconstruction. As the claim of the theorem is asymptotic, we assume that $n$ is sufficiently large. We prove the theorem in two cases.

First suppose $W_k\geq\binom{n}{k}$. By Lemma~\ref{lem:Wkbounds}, for sufficiently large $n$ we have $W_k\geq 0.685n$, which is stronger than what we need.

Now suppose that $W_k\leq\binom{n}{k}$. We can assume $k\geq 4$, as otherwise reconstruction is not possible. We obtain \begin{equation*} W_n\leq \binom{\binom{n}{k} +W_k -1}{W_k-1}\leq \binom{2\binom{n}{k}}{W_k-1}\leq \binom{2\binom{n}{k}}{W_k} \leq n^{kW_k}.
\end{equation*}
Note that in the inequality, for large $n$, $k$ cannot be bounded by a constant.
 If $k$ were bounded, then $W_n \in O(n^{const})$, contradicting Equation~\ref{eqn:Wnboundsb}. 
Taking the logarithm of $C_1 \frac{b^n}{n^{3/2}} \le W_n \le n^{k W_k}$, we get
\begin{equation*}
n\log b - \frac{3}{2}\log n +\log C_1\leq \log W_n \leq kW_{k} \log n. 
\end{equation*}
Using Equation~\eqref{eqn:Wnboundsb} and taking the logarithm again gives
\begin{equation*}
\log n  + O(1)\leq \log\log W_n \leq \log k+\log\log n +\log C_2 + k\log b-\frac{3}{2}\log k. \end{equation*} Now reordering gives \[k\geq \frac{\log n  +\frac{1}{2}\log k-\log\log n  -O(1)}{\log b} \ge (1-o(1)) \frac{\log n}{\log b}. \qedhere \] 
\end{proof}

\subsection{Upper bound}

We now consider an upper bound on the deck and multideck size required for reconstruction. First we establish some lemmas relating the composition of a binary tree to the structure of the subtrees in its deck.

\begin{lemma}\label{lm:largedecomp}
    Suppose $T = T_1 \oplus T_2$ where $|T_i| = n_i$, $i \in \{1,2\}$, and $2 \le n_1 \le n_2$. Let $D_i =\{ T_i\oplus T^{\star}: T^{\star}\in \mathcal{D}(T_{3-i})\}$. Then 
    \begin{enumerate}
    \item $\mathcal{D}(T)=D_1\cup D_2$. 
    \item $D_1\cap D_2\ne\emptyset$ if and only if $T_1 = T_2$ if and only if $D_1=D_2$.
    \item If $n_1=2$ and $n_2\ge 3$, then $D_2 = \{C_1 \oplus T_2\}$ and $C_1 \oplus T_2$ is the unique tree in $\mathcal{D}(T)$ of the form $C_1\oplus T^{\star}$.
    \end{enumerate}
\end{lemma}

\begin{proof}
    Statement~1 follows from the definitions and Statement~3 follows from Statement~2 and the definitions. For Statement~2, suppose $T_1 \oplus T^* \in D_1 \cap D_2$. Since $T^* \in \mathcal{D}(T_2)$, we cannot have $T^* = T_2$. Thus $T_2 = T_1$ and so $D_1 = D_2$. Last, if $D_1 = D_2$ we have $D_1 \cap D_2 \neq \emptyset$.
\end{proof}

The following two lemmas are obvious from the definitions
\begin{lemma}\label{lm:large} If $T$ is composed of two trees of size $n_1$ and $n_2$ where $2\le n_1\le n_2$ then the root-splits of the trees in the deck are $\{n_1-1,n_2\},\{n_1,n_2-1\}$.
\end{lemma}

\begin{lemma}\label{lm:smalldecomp} 
Suppose $T = C_1 \oplus T_2$, where $T_2$ is a binary tree of size $n-1$ with $n \ge 3$. Let
$D_2= \{C_1\oplus T^{\star}:T^{\star}\in \mathcal{D}(T_2)\}$. Then
 $\mathcal{D}(T)=\{T_2\}\cup D_2$ and $T_2\in D_2$ if and only if $T_2$ has a leaf at distance $1$ from its root.
Consequently, one of the rootsplits of the trees in the deck of $T$ is $\{1,n_2-1\}$ and the other (not necessarily different)
is the root-split of $T_2$.
\end{lemma}

Recall that $R(4) = R^{(m)}(4) = 4$ and $R(5) = 5$ while $R^{(m)}(5) = 4$. For $n \ge 6$ we next show that the size-$n$ rooted binary trees are determined by their decks. It follows that they are determined by their multidecks.

\begin{theorem}\label{th:deck} 
If $n \notin \{4,5\}$, then $R(n) \le n-1$. 
\end{theorem} 

\begin{figure}[!ht]
\centering
\begin{tikzpicture}[scale=.6]
\node[] at (14,3) {}; 
       \node[fill=black,rectangle,inner sep=2pt]   at (11,0) {};
   \node[fill=black,rectangle,inner sep=2pt]   at (12,0) {};
    \node[fill=black,rectangle,inner sep=2pt]   at (13,0) {};      
    \node[fill=black,rectangle,inner sep=2pt]   at (14,0) {};     
    \node[fill=black,rectangle,inner sep=2pt]   at (15,0) {};
    \node[fill=black,rectangle,inner sep=2pt]   at (16,0) {};      
          \node[fill=black,circle,inner sep=1pt]  at (11.5,0.5) {};
          \node[fill=black,circle,inner sep=1pt]  at (12,1) {};
             \node[fill=black,circle,inner sep=1pt]  at (12.5,1.5) {};
             \node[fill=black,circle,inner sep=1pt]  at (13,2) {};   
              \node[fill=black,circle,inner sep=1pt]  at (13.5,2.5) {};               
          \draw (11,0)--(11.5,.5)--(12,0);
           \draw (11.5,.5)--(12,1)--(13,0);
          \draw (12,1)--(12.5,1.5)--(14,0);
            \draw (12.5,1.5)--(13,2)--(15,0);
            \draw (13,2)--(13.5,2.5)--(16,0);
\node at (13.5,-1) {$(6,0,0)\rightarrow (1,0,0)$};
\end{tikzpicture}
\quad
\begin{tikzpicture}[scale=.6]
      \node[fill=black,rectangle,inner sep=2pt]   at (19,0) {};
   \node[fill=black,rectangle,inner sep=2pt]   at (20,0) {};
    \node[fill=black,rectangle,inner sep=2pt]   at (21,0) {};      
    \node[fill=black,rectangle,inner sep=2pt]   at (22,0) {};    
    \node[fill=black,rectangle,inner sep=2pt]   at (23,0) {};   
    \node[fill=black,rectangle,inner sep=2pt]   at (24,0) {};                       
          \node[fill=black,circle,inner sep=1pt]  at (19.5,0.5) {};
          \node[fill=black,circle,inner sep=1pt]  at (21.5,.5) {};
             \node[fill=black,circle,inner sep=1pt]  at (20.5,1.5) {};
              \node[fill=black,circle,inner sep=1pt]  at (21,2) {};   
              \node[fill=black,circle,inner sep=1pt]  at (21.5,2.5) {};                                 
          \draw (19,0)--(19.5,.5)--(20,0);
           \draw (19.5,.5)--(20.5,1.5)--(21.5,0.5);
          \draw (21,0)--(21.5,.5)--(22,0);
         \draw (20.5,1.5)--(21,2)--(23,0);
         \draw (21,2)--(21.5,2.5)--(24,0);
\node at (21.5,-1) {$(4,2,0)\rightarrow (1,1,0)$};
\end{tikzpicture}
\quad
\begin{tikzpicture}[scale=.6]
     \node[fill=black,rectangle,inner sep=2pt]   at (32,0) {};
     \node[fill=black,rectangle,inner sep=2pt]   at (30,0) {};
   \node[fill=black,rectangle,inner sep=2pt]   at (31,0) {};
          \node[fill=black,circle,inner sep=1pt]  at (30.5,0.5) {};
          \draw (30,0)--(30.5,.5)--(31,0);            
  \node[fill=black,rectangle,inner sep=2pt]   at (27,0) {};
   \node[fill=black,rectangle,inner sep=2pt]   at (28,0) {};
    \node[fill=black,rectangle,inner sep=2pt]   at (29,0) {};      
          \node[fill=black,circle,inner sep=1pt]  at (27.5,0.5) {};
          \node[fill=black,circle,inner sep=1pt]  at (28,1) {};
             \node[fill=black,circle,inner sep=1pt]  at (29,2) {};  
              \node[fill=black,circle,inner sep=1pt]  at (29.5,2.5) {};           
          \draw (27,0)--(27.5,.5)--(28,0);
           \draw (27.5,.5)--(28,1)--(29,0);
       \draw (28,1)--(29,2)--(30.5,0.5);
       \draw (29,2)--(29.5,2.5)--(32,0);
\node at (29.5,-1) {$(2,1,3)\rightarrow (1,1,1)$};   
\end{tikzpicture} 
\\
\begin{tikzpicture}[scale=.6]
\node[] at (14,3) {}; 
       \node[fill=black,rectangle,inner sep=2pt]   at (11,0) {};
   \node[fill=black,rectangle,inner sep=2pt]   at (12,0) {};
    \node[fill=black,rectangle,inner sep=2pt]   at (13,0) {};      
    \node[fill=black,rectangle,inner sep=2pt]   at (14,0) {};     
    \node[fill=black,rectangle,inner sep=2pt]   at (15,0) {};
    \node[fill=black,rectangle,inner sep=2pt]   at (16,0) {};      
          \node[fill=black,circle,inner sep=1pt]  at (11.5,0.5) {};
          \node[fill=black,circle,inner sep=1pt]  at (12,1) {};
             \node[fill=black,circle,inner sep=1pt]  at (12.5,1.5) {};
             \node[fill=black,circle,inner sep=1pt]  at (15.5,.5) {};   
              \node[fill=black,circle,inner sep=1pt]  at (13.5,2.5) {};               
          \draw (11,0)--(11.5,.5)--(12,0);
           \draw (11.5,.5)--(12,1)--(13,0);
          \draw (12,1)--(12.5,1.5)--(14,0);
            \draw (12.5,1.5)--(13.5,2.5)--(15.5,0.5);
            \draw (15,0)--(15.5,0.5)--(16,0);
\node at (13.5,-1) {$(2,0,4)\rightarrow (1,0,1)$};
\end{tikzpicture}
\quad
\begin{tikzpicture}[scale=.6]
\node[] at (14,3) {}; 
       \node[fill=black,rectangle,inner sep=2pt]   at (11,0) {};
   \node[fill=black,rectangle,inner sep=2pt]   at (12,0) {};
    \node[fill=black,rectangle,inner sep=2pt]   at (13,0) {};      
    \node[fill=black,rectangle,inner sep=2pt]   at (14,0) {};     
    \node[fill=black,rectangle,inner sep=2pt]   at (15,0) {};
    \node[fill=black,rectangle,inner sep=2pt]   at (16,0) {};      
          \node[fill=black,circle,inner sep=1pt]  at (11.5,0.5) {};
          \node[fill=black,circle,inner sep=1pt]  at (13.5,.5) {};
             \node[fill=black,circle,inner sep=1pt]  at (12.5,1.5) {};
             \node[fill=black,circle,inner sep=1pt]  at (15.5,.5) {};   
              \node[fill=black,circle,inner sep=1pt]  at (13.5,2.5) {};               
          \draw (11,0)--(11.5,.5)--(12,0);
           \draw (14,0)--(13.5,.5)--(13,0);
          \draw (11.5,.5)--(12.5,1.5)--(13.5,0.5);
            \draw (12.5,1.5)--(13.5,2.5)--(15.5,0.5);
            \draw (15,0)--(15.5,0.5)--(16,0);
\node at (13.5,-1) {$(0,2,4)\rightarrow (0,1,1)$};
\end{tikzpicture}
\quad
\begin{tikzpicture}[scale=.6]
     \node[fill=black,rectangle,inner sep=2pt]   at (32,0) {};
     \node[fill=black,rectangle,inner sep=2pt]   at (30,0) {};
   \node[fill=black,rectangle,inner sep=2pt]   at (31,0) {};
          \node[fill=black,circle,inner sep=1pt]  at (30.5,0.5) {};
          \draw (30,0)--(30.5,.5)--(31,0);            
  \node[fill=black,rectangle,inner sep=2pt]   at (27,0) {};
   \node[fill=black,rectangle,inner sep=2pt]   at (28,0) {};
    \node[fill=black,rectangle,inner sep=2pt]   at (29,0) {};      
          \node[fill=black,circle,inner sep=1pt]  at (27.5,0.5) {};
          \node[fill=black,circle,inner sep=1pt]  at (28,1) {};
             \node[fill=black,circle,inner sep=1pt]  at (31,1) {};  
              \node[fill=black,circle,inner sep=1pt]  at (29.5,2.5) {};           
          \draw (27,0)--(27.5,.5)--(28,0);
           \draw (27.5,.5)--(28,1)--(29,0);
       \draw (32,0)--(31,1)--(30.5,0.5);
       \draw (28,1)--(29.5,2.5)--(31,1);
\node at (29.5,-1) {$(0,0,6)\rightarrow (0,0,1)$};   
\end{tikzpicture} 
\caption{Binary trees of size $6$. The triplets $(i,j,k)\rightarrow(\ell,m,n)$ are the multiplicities of the three size-$5$ trees
$C_5$, $C_1\oplus B_2$, and $C_2\oplus C_3$ respectively in the multideck and then the deck of the tree.}\label{fig:6trees}
\end{figure}
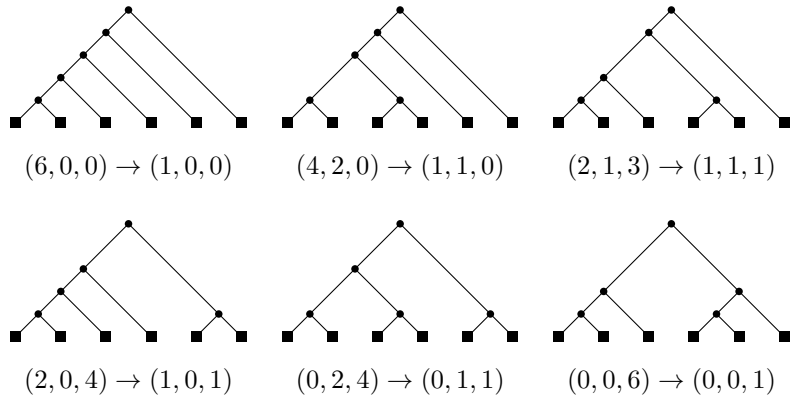

\begin{proof} 
As $R(n)$ is only defined for $n \geq 4$, we only need to show the statement for $n\ge 6$. One can directly verify the statement for the six trees of size 6 -- see Figure~\ref{fig:6trees}. 

Let $T$ be a size-$n$ tree with $n > 6$. Suppose $T = T_1 \oplus T_2$, where $|T_i| = n_i$ and $n_1 \le n_2$. There are three possibilities for $n_1$ and $n_2$. If $n_1 = n_2$, then $n_1 \ge 4$ and all trees in $\mathcal{D}(T)$ have root-split $\{n_1-1, n_1\}$ by Lemma~\ref{lm:large}. If instead $2 \le n_1 < n_2$, again by Lemma~\ref{lm:large}, the trees in $\mathcal{D}(T)$ have two different root-splits $\{n_1-1, n_2\}$ and $\{n_1, n_2 -1\}$. Finally, if $n_1 = 1$, then by Lemma~\ref{lm:smalldecomp}, the trees in $\mathcal{D}(T)$ have root-splits $\{1, n_2-1\}$ and $\{\ell, n_2 - \ell\}$ where the latter is the root-split of $T_2$ and we assume that $\ell \le n_2 - \ell$. Suppose that the root-splits of the elements of $\mathcal{D}(T)$ are $\{a, n-1-a\}$ and $\{b, n-1-b\}$ where $1 \le a \le b \le \frac{n-1}{2}$. 

\textbf{Case 1:} $a \ge 2$. Then we have $n_1 = a+1 > 2$ and so $n_2 = n-1-a$ and $\{b, n-1-b\} = \{n_1, n_2 -1\}$. 

If $a=b$, then $n_1 = n_2$ and the root-splits are all $\{n_1-1, n_1\}$. For each $T' \in \mathcal{D}(T)$, we have $T' = T_1' \oplus T_2'$ with $|T_1'| < |T_2'|$ so that $T_2' \in \{T_1, T_2\}$. Thus, by looking at all elements of $\mathcal{D}(T)$ we can identify $T_1$ and $T_2$ and reconstruct $T$. 

If instead $a < b$, then we have $3 \le n_1 < n_2$ and $b = n_1$. Then there exists $T' \in \mathcal{D}(T)$ such that $T' = T_1' \oplus T_2'$ with $|T_1'| = a$. Then $T_2' = T_2$. 

Let $D = \{T' \in \mathcal{D}(T): T' = T_1' \oplus T_2', |T_1'| = b\}$. Then $D = \{P \oplus T_1: P \in \mathcal{D}(T_2)\}$. If $T_1 \in \mathcal{D}(T_2)$, then $D$ contains a unique tree composed of isomorphic trees, namely $T_1 \oplus T_1$, and we can thus recognize $T_1$. Otherwise, $T_1 \notin \mathcal{D}(T_2)$ and we can recognize $T_1$ from any element of $D$. 

\textbf{Case 2:} $a = 1$ and $b \neq 2$. Then $n_1 = 1$, $T_1 = C_1$, $n_2 = n-1$ and $\{b, n-1-b\}$ is root-split of $T_2$. 

If $b > 2$, then there is a unique tree in $\mathcal{D}(T)$ with root split $\{b, n-1-b\}$ and that tree is $T_2$. 

Suppose $b = 1$, then $T_2 = C_1 \oplus T_3$ for some tree $T_3$. Then either $T = C_n$ or $T = (C_1 \oplus)^i (T' \oplus T'')$ for some $i \ge 2$ and trees $T'$ and $T''$ each with at least two leaves. If $T = C_n$, then $\mathcal{D}$ contains a single tree $C_{n-1}$. Otherwise $\mathcal{D}(T)$ contains trees of the forms $(C_1 \oplus)^{i-1} (T' \oplus T'')$ and $(C_1 \oplus)^j T'''$ for some $j \ge i$ and tree $T'''$. In the former case, we can recognize $T = C_n$ and in the latter, the cards of the form $(C_1 \oplus)^{i-1} (T' \oplus T'')$ allow us to recognize $T_2$ and reconstruct $T$.

\textbf{Case 3:} $a = 1$ and $b = 2$. Then $n_1 \in \{1,2\}$ and $T_1=C_{n_1}$. Set $C = \{C_1 \oplus Q: |Q| = n-1\}$. We claim that we can determine $n_1$ (and consequently $T_1$) from $|\mathcal{D}(T) \cap C|$. If $n_1 = 1$, then $T_1 = C_1$ and $T_2 = C_2 \oplus T_3$ for some tree $T_3$. Then $\mathcal{D}(T) = \{C_2 \oplus T_3, C_1 \oplus (C_1 \oplus T_3)\} \cup \{C_1 \oplus (C_2 \oplus T') : T' \in \mathcal{D}(T_3)\}$. Hence $|\mathcal{D}(T) \cap C| \ge 2$. If instead $n_1 = 2$, then $T_1 = C_2$ and $\mathcal{D}(T) = \{C_1 \oplus T_2\} \cup \{C_2 \oplus T': T' \in \mathcal{D}(T_2)\}$ by Lemma~\ref{lm:largedecomp}. Hence $|\mathcal{D} \cap C| = 1$. 

In the case that $n_1 = 1$, $\mathcal{D}(T) \setminus C = \{T_2\}$. In the case $n_1 = 2$, $\mathcal{D}(T) \cap C = \{C_1 \oplus T_2\}$, which allows us to recognize $T_2$. In either case we can then reconstruct $T$. 
\end{proof}

\begin{corollary}\label{th:multi} If $n\ne 4$, then $R^{(m)}(n) \le n-1$. 
\end{corollary}

\begin{proof} We already determined this for $n\le 5$, and the rest follows from Theorem~\ref{th:deck}.
\end{proof}

\section{Problems related to deck cardinalitites} \label{sec:non-isom subtrees}       
In this section, we consider questions related to the maximum and minimum cardinality of a size-$i$ deck of a size-$n$ tree, as well as the maximum and minimum number of subtrees of size-$n$ trees. Note that the number of subtrees of a tree $T$ is $S(T) = \sum_{i=1}^{|T|}|\D_i(T)|$. 

We start by considering trees with few subtrees. Since for all $i \in [n]$, we have $\D_i(C_n) = \{C_i\}$, the minimum cardinality of a size-$i$ deck of a size-$n$ tree is one and $\min_{T: |T|=n} S(T) = n$. In Section~\ref{subsec: few subtrees}, we characterize the trees that have only one size-$i$ subtree and show that if $S(T) = n$, then either $T = C_n$ for some $n$ or $T = B_2$. 

We then turn to trees with many subtrees, where we denote the maximum number of subtrees of size-$n$ trees by $S(n)$, that is, $S(n)=\max_{T:|T|=n}\sum_{i=1}^n |\D_i(T)|$. In Section~\ref{subsec:many subtrees}, we obtain upper and lower bounds for $S(n)$. More precisely, we show that $S(n) \in O(1.8648^n)$ (Theorem~\ref{thm:Snupperbound}) and $S(n) \in \Omega(5^{n/4})$ (Theorem~\ref{thm:Snlowerbound}).

We end by briefly discussing \emph{$k$-universal} trees, that is, trees that contain all size-$k$ trees in their size-$k$ deck, i.e., $\vert \D_k(T) \vert = W_k$. We use $u(k)$ to denote the minimum number of leaves in a $k$-universal tree and recall lower and upper bounds for $u(k)$ from the literature in Section~\ref{subsec:universal}. We complement this with an exhaustive enumeration of $k$-universal trees for small $k$. 

We remark that the results in Sections~\ref{subsec:many subtrees} and~\ref{subsec:many subtrees} are related in the following way. If $(i,n)$ is a pair of integers such that $n\ge u(i)$, then the maximum cardinality of a size-$i$ deck in the size-$n$ trees is $W_i$. For $i\in[3]$, we have $W_i=1=\D_i(T)$ for any $T$ with $|T|\ge i$, and for $n\ge 5$ we show that $|\D_4(T)|=2=W_4$ for any $T \neq C_n$, so the maximum cardinality of a size-$i$ deck among the size-$n$ trees is undetermined only for $5\le i< u(i)$.

\subsection{Trees containing few subtrees} \label{subsec: few subtrees}

First we consider trees with minimum-size decks. 

\begin{lemma}\label{lm:cater} Assume that $T=C_1\oplus T_2$ for
some rooted binary tree $T_2$ and $|\mathcal{D}(T)|=1$. Then
$T=C_n$ where $n$ is the size of $T$.
\end{lemma}

\begin{proof}
Let $n$ be the size of $T$; we will prove the statement by induction on $n$.
When $n=2$, it must be that $T= C_2$. Assume $n>2$ and the statement is true for trees of size $n-1$. By Lemma~\ref{lm:smalldecomp} we have that $T_2=C_1\oplus T_3$
for some tree $T_3$. Furthermore, $|\mathcal{D}(T_2)|=1$, so by the induction hypothesis $T_2= C_{n-1}$. Therefore $T=C_n$.
\end{proof}

\begin{definition} For $k,\ell\in\mathbb{N}$, where $\ell \ge 2$, we define the \emph{$(k,\ell)$-jellyfish $J_{k,\ell}$} as follows: take a height-$k$ rooted complete binary tree $B_k$ and identify each of its leaves with the root of a caterpillar $C_{\ell}$.
\end{definition}

An example is shown in Figure \ref{fig:Ckl}. We have $|J_{k,\ell}| = 2^k\ell$. Observe that $J_{k,2}=B_{k+1}$ and $J_{0,\ell}=C_{\ell}$.

\begin{figure}
\centering
\scalebox{0.75}{
\begin{tikzpicture}
    \pic at (0,0) {tree4.1};
    \pic at (4,0) {tree4.1};
    \draw (1.5,1.5) -- (3.5,3.5) -- (5.5,1.5);
    \node[fill=black,circle,inner sep=1pt]  at (3.5,3.5) {};
\end{tikzpicture}
}
\caption{The tree $J_{1,4}$.}
\label{fig:Ckl}
\end{figure}
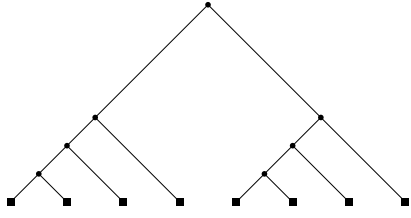

\begin{theorem}\label{th:singleD} If $T$ is a rooted binary tree, then
$|\mathcal{D}(T)|=1$ if and only if $T= J_{k,\ell}$ for some $(k,\ell)$.
\end{theorem}
\begin{proof}
First we show that that $|\D(J_{k,\ell})| = 1$. 

Since $J_{0,\ell}= C_{\ell}$, for $\ell\ge 2$ we have
$\mathcal{D}(J_{0,\ell})=\D(C_{\ell})=\{C_{\ell-1}\}$,
so $|\D(J_{0,\ell})|=1$ for all $\ell\ge 2$. 
Consider $J_{k,\ell}$ such that $k>0$ and $\ell \ge 2$. The vertices at distance at most $k$ from the root form a complete binary tree $B_k$. Let $v_1,\ldots, v_{2^k}$ be an enumeration of the vertices of $J_{k,\ell}$ at distance $k$ from the root. Removing a leaf from $J_{k,\ell}$ is removing a leaf $u_i$ of one of the $C_{\ell}$-s rooted at a some $v_i$, resulting in a $C_{\ell-1}$ rooted at that $v_i$. If $x,y$ are leaves of $B_k$ then there is an automorphism of $B_k$ that maps $x$ to $y$. We can extend this automorphism to an isomorphism between $J_{k,\ell} - u_i$ and $J_{k,\ell}-u_j$ for any choice of leaves $u_i$ and $u_j$. Hence $|\D(J_{k,\ell})| = 1$. 

Next we show that if $T$ is a tree such that $|\D(T)|=1$, then $T= J_{k,\ell}$.

Assume $T$ is a tree with $|\D(T)|=1$.
Since $C_1$ is the only tree of size 1 and $\D(C_1) = \emptyset$, if $|\D(T)|=1$,
then $|T|\ge 2$. 
We will show there are $(k,\ell)$ with $k \ge 0, \ell \ge 2$ 
such that $J_{k,\ell}= T$. 

Since $|T|\ge 2$, the root has $2$ children. Let $k$ be maximal such that every node at distance at most $k$ from the root of $T$ has exactly two
children, and label the nodes at distance $k$ with $v_1,\ldots,v_{2^k}$.
Then the paths from the root of $T$ to each of the children of $v_1,\ldots, v_{2^k}$ form a $B_{k+1}$ subgraph 
of $T$, and there is a $t$ such that one of the children of $v_{t}$ is a leaf of $T$.
For $i\in[2^k]$ let $T_i$ denote the subtree of $T$
rooted at $v_i$ containing all nodes that are descendants of $v_i$.
Each $T_i$ has at least $2$ leaves and $T_{t}= C_1\oplus T_{t}^{\star}$ for some tree $T_t^*$.

Suppose $|T_i| \neq |T_j|$ for some $i,j$. Then removing a leaf of $T_i$ and removing a leaf of $T_j$ from $T$ must result in different trees, contradicting our assumption on $|\D(T)|$. Hence all $T_i$ have the same size $m$, where $m\ge 2$.

Suppose for $i$ we have that $T_i = A \oplus B$ where both $A$ 
and $B$ have size at least $2$. Then removing a leaf of $T_i$ 
from $T$ and removing the leaf-neighbor of $v_{t}$ results in different trees, contradicting our assumption on $|\D(T)|$. Thus, all $T_i$ are of the form $C_1\oplus T_{i}^{\star}$ for some $T_i^{\star}$ of size $m-1$. Since our assumption on $\D(T)$ implies that $|\D(T_i)|=1$, Lemma~\ref{lm:cater} yields $T_i= C_{m}$. Therefore $T= J_{k,m}$. 
\end{proof}

We next consider trees with few subtrees. We have that $S(C_n) = n$. In addition, the four-leaf tree $B_2$ also contains exactly four induced rooted binary trees. We show that caterpillars and $B_2$ are the only trees that achieve $\min_{T: |T|=n} S(T)$. Since the only trees of size at most three are caterpillars, we need only further discuss trees of size at least 5.

\begin{lemma}\label{lem:size1_4deck}
Let $T$ be a tree of size $n \ge 5$. Then $C_4 \in \D_4(T)$. Additionally, $\D_4(T) = \{C_4\}$ if and only if $T =C_n$. 
\end{lemma}
\begin{proof}
Let $T = T_1 \oplus T_2$, where we assume without loss of generality that $|T_1| \le |T_2|$.  In particular, $|T_2| \ge 3$. Any induced subtree formed from one leaf of $T_1$ and three leaves of $T_2$ is isomorphic to $C_4$. Hence, $C_4\in \mathcal{D}_4(T)$.

If $T = C_n$, then $\D_4(T) = \{C_4\}$. To finish the proof, we use induction to show that if $T$ is a tree of size $n \ge 4$ and $\D_4(T) = \{C_4\}$, then $T= C_n$. If $|T| = 4$, then $\D_4(T) = \{T\}$, so $T = C_4$. Assume $n \ge 5$ and the statement is true for trees of size $n-1$. Let $T = T_1 \oplus T_2$. If $T_1$ and $T_2$ have more than one leaf, then the induced subtree of $T$ corresponding to two leaves from each of $T_1$ and $T_2$ is isomorphic to $B_2$. Hence, we can assume $T_1 = C_1$. Then $|T_2| = n-1 \ge 4$ so that $\D_4(T_2)$ is nonempty. But then $\D_4(T_2) = \{C_4\}$ and by the induction hypothesis, $T_2 = C_{n-1}$. Hence $T = C_n$. 
\end{proof}

Lemma~\ref{lem:size1_4deck} implies that if $n \ge 5$ and $T \neq C_n$, then $\vert \D_4(T) \vert = 2 = W_4$. On the other hand, if $n \ge 5$ and $|\D_4(T)| = 1$, then $T = C_n$.
We can extend this res result when $|\D_j(T)| = 1$ for larger $j$. 

\begin{lemma}\label{lem:size1_anydeck}
Let $T$ be a tree of size $n\geq 7$ and suppose $|\mathcal{D}_j(T)|=1$ for some $5\leq j\leq n-2$. Then $T= C_n$ and $|\mathcal{D}_j(T)|=1$ for all $1\leq j\leq n$. 
\end{lemma}
\begin{proof}
It suffices to show $T=C_n$, as $|\mathcal{D}_j(C_n)|=1$ for all $1\leq j\leq n$. 

Suppose that $\mathcal{D}_{j+1}(T)$ contains two distinct trees $T_1$ and $T_2$. Since $j+1 \ge 6$, by Theorem~\ref{th:deck}, we have $T_1$ and $T_2$ have distinct size-$j$ decks. This contradicts our assumption that $|D_j(T)| = 1$. Thus, we have $|D_{j+1}(T)| = 1$. Furthermore, $|\mathcal{D}_{k}(T)|=1$ for all $j\leq k\leq n-1$. By Theorem~\ref{th:singleD}, $T= J_{k,\ell}$ for some $k \ge 0$ and $\ell \ge 2$. 
If $k \ge 1$, then removing two leaves from $J_{k,\ell}$ from the same $C_\ell$ or distinct $C_\ell$-s results in non-isomorphic trees. Thus $|D_{n-2}(T)| = 1$ implies $k = 0$ and $T = J_{0,n} = C_n$. 
\end{proof}

\begin{theorem} \label{thm:size1jdeck}
Let $T$ be a tree of size $n \ge 5$ and $|D_j(T) |= 1$ for some $4 \le j \le n-1$. Then $T = C_n$, or $j=n-1$ and $T = J_{k,\ell}$ for some $k \ge 0, \ell \ge 2$.
\end{theorem}

\begin{proof}
In the case $j=n-1$ and $n\geq 5$, Theorem~\ref{th:singleD} implies $T=J_{k,\ell}$ for some $k \ge 0,\ell \ge 2$. In the case $4\leq j< n-2$, Lemmas~\ref{lem:size1_4deck} and~\ref{lem:size1_anydeck} imply $T = C_n$.
\end{proof}

For a tree to achieve $\min_{T:|T| = n} S(T)$, we require $|D_j(T)|= 1$ for all $1 \le j \le n$. From Theorem~\ref{thm:size1jdeck}, we get the following corollary.

\begin{corollary}
For any $n$, we have $\min_{T:|T| = n} S(T) = n$. Equality is achieved only for $B_2$ and caterpillars $C_n$.
\end{corollary}

\subsection{Trees containing many subtrees} \label{subsec:many subtrees}

We now consider trees with many nonisomorphic subtrees. We start by considering trees with decks that have maximum cardinality. 

\begin{definition}
For $n \ge 1$, we recursively define trees $Z_n$. For $1 \le n \le 3$, let $Z_n = C_n$. For $n > 3$, 
\[ Z_n = \begin{cases}
C_1\oplus Z_{n-1} &\text{if } n \equiv 0  \, (\textup{mod}\ 3) \\
C_2\oplus Z_{n-2} &\text{otherwise.}
\end{cases}. \]
\end{definition}

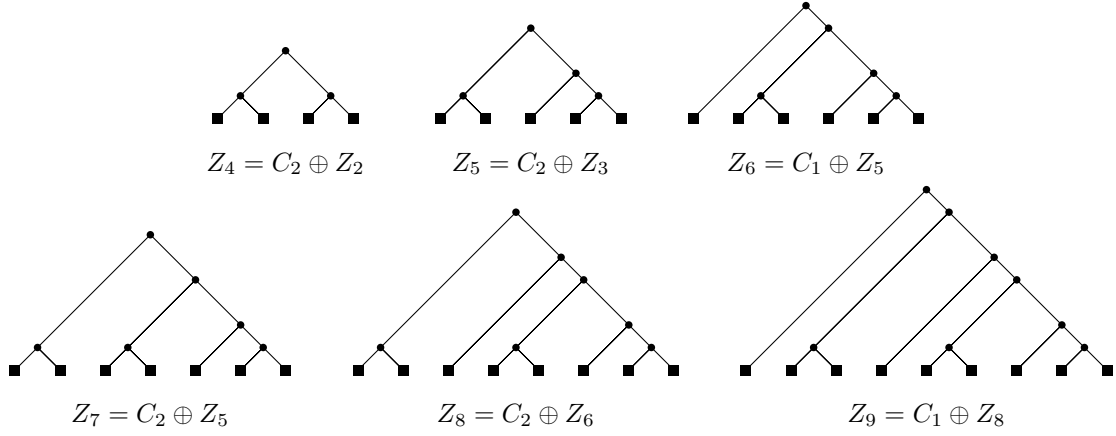
\begin{figure}
    \centering
\begin{tikzpicture}[scale=0.6]
        \draw (2,0)-- (2.5,.5)--(3,0) -- (2.5,.5)--(3.5,1.5)-- (4.5,0.5) -- (4,0) -- (4.5,0.5) -- (5,0);  
        \node[fill=black,rectangle,inner sep=2pt]  at (2,0) {}; 
        \node[fill=black,rectangle,inner sep=2pt]  at (3,0) {}; 
        \node[fill=black,rectangle,inner sep=2pt]  at (4,0) {}; 
        \node[fill=black,rectangle,inner sep=2pt]  at (5,0) {};   
        \node[fill=black,circle,inner sep=1pt]  at (3.5,1.5) {};
        \node[fill=black,circle,inner sep=1pt]  at (4.5,0.5) {};
        \node[fill=black,circle,inner sep=1pt]  at (2.5,0.5) {};
        \node at (3.5,-1) {$Z_4=C_2\oplus Z_2$};
    \end{tikzpicture}
    \qquad                
\begin{tikzpicture}[scale=0.6]
        \draw (3,2) -- (1,0) -- (1.5,0.5) -- (2,0) -- (1.5,0.5) -- (3,2) -- (4,1) -- (3,0) -- (4,1) -- (4.5,0.5) -- (4,0) -- (4.5,0.5) -- (5,0);  
        \node[fill=black,rectangle,inner sep=2pt]  at (1,0) {};      
        \node[fill=black,rectangle,inner sep=2pt]  at (2,0) {}; 
        \node[fill=black,rectangle,inner sep=2pt]  at (3,0) {}; 
        \node[fill=black,rectangle,inner sep=2pt]  at (4,0) {}; 
        \node[fill=black,rectangle,inner sep=2pt]  at (5,0) {};   
        \node[fill=black,circle,inner sep=1pt]  at (3,2) {};
        \node[fill=black,circle,inner sep=1pt]  at (4.5,0.5) {};\node[fill=black,circle,inner sep=1pt]  at (1.5,0.5) {};
        \node[fill=black,circle,inner sep=1pt]  at (4,1) {};
         \node at (3,-1) {$Z_5=C_2\oplus Z_3$};           
    \end{tikzpicture}
    \qquad         
    \begin{tikzpicture}[scale=0.6]
        \draw (0,0) -- (2.5,2.5) -- (3,2) -- (1,0) -- (1.5,0.5) -- (2,0) -- (1.5,0.5) -- (3,2) -- (4,1) -- (3,0) -- (4,1) -- (4.5,0.5) -- (4,0) -- (4.5,0.5) -- (5,0);
        \node[fill=black,rectangle,inner sep=2pt]  at (0,0) {};      
        \node[fill=black,rectangle,inner sep=2pt]  at (1,0) {};      
        \node[fill=black,rectangle,inner sep=2pt]  at (2,0) {}; 
        \node[fill=black,rectangle,inner sep=2pt]  at (3,0) {}; 
        \node[fill=black,rectangle,inner sep=2pt]  at (4,0) {}; 
        \node[fill=black,rectangle,inner sep=2pt]  at (5,0) {};   
        \node[fill=black,circle,inner sep=1pt]  at (2.5,2.5) {};
        \node[fill=black,circle,inner sep=1pt]  at (3,2) {};
        \node[fill=black,circle,inner sep=1pt]  at (4.5,0.5) {};\node[fill=black,circle,inner sep=1pt]  at (1.5,0.5) {};
        \node[fill=black,circle,inner sep=1pt]  at (4,1) {};
         \node at (2.5,-1) {$Z_6=C_1\oplus Z_5$};           
    \end{tikzpicture}
\\
    \begin{tikzpicture}[scale=0.6]
        \draw (-1,0) -- (-.5,.5)--(0,0)--(-.5,.5)--(2,3) -- (2.5,2.5) -- (3,2) -- (1,0) -- (1.5,0.5) -- (2,0) -- (1.5,0.5) -- (3,2) -- (4,1) -- (3,0) -- (4,1) -- (4.5,0.5) -- (4,0) -- (4.5,0.5) -- (5,0);
        \node[fill=black,rectangle,inner sep=2pt]  at (-1,0) {};  
        \node[fill=black,rectangle,inner sep=2pt]  at (0,0) {};      
        \node[fill=black,rectangle,inner sep=2pt]  at (1,0) {};      
        \node[fill=black,rectangle,inner sep=2pt]  at (2,0) {}; 
        \node[fill=black,rectangle,inner sep=2pt]  at (3,0) {}; 
        \node[fill=black,rectangle,inner sep=2pt]  at (4,0) {}; 
        \node[fill=black,rectangle,inner sep=2pt]  at (5,0) {};   
        \node[fill=black,circle,inner sep=1pt]  at (-.5,.5) {};
        \node[fill=black,circle,inner sep=1pt]  at (2,3) {};
        \node[fill=black,circle,inner sep=1pt]  at (3,2) {};
        \node[fill=black,circle,inner sep=1pt]  at (4.5,0.5) {};\node[fill=black,circle,inner sep=1pt]  at (1.5,0.5) {};
        \node[fill=black,circle,inner sep=1pt]  at (4,1) {};
          \node at (2,-1) {$Z_7=C_2\oplus Z_5$};          
    \end{tikzpicture}
    \qquad 
    \begin{tikzpicture}[scale=0.6]
        \draw (-2,0) -- (-1.5,0.5) -- (-1,0) -- (-1.5,0.5) -- (1.5,3.5) -- (2.5,2.5) -- (0,0) -- (2.5,2.5) -- (3,2) -- (1,0) -- (1.5,0.5) -- (2,0) -- (1.5,0.5) -- (3,2) -- (4,1) -- (3,0) -- (4,1) -- (4.5,0.5) -- (4,0) -- (4.5,0.5) -- (5,0);
        \node[fill=black,rectangle,inner sep=2pt]  at (-2,0) {};      
        \node[fill=black,rectangle,inner sep=2pt]  at (-1,0) {};  
        \node[fill=black,rectangle,inner sep=2pt]  at (0,0) {};      
        \node[fill=black,rectangle,inner sep=2pt]  at (1,0) {};      
        \node[fill=black,rectangle,inner sep=2pt]  at (2,0) {}; 
        \node[fill=black,rectangle,inner sep=2pt]  at (3,0) {}; 
        \node[fill=black,rectangle,inner sep=2pt]  at (4,0) {}; 
        \node[fill=black,rectangle,inner sep=2pt]  at (5,0) {};   
        \node[fill=black,circle,inner sep=1pt]  at (2.5,2.5) {};
        \node[fill=black,circle,inner sep=1pt]  at (3,2) {};
        \node[fill=black,circle,inner sep=1pt]  at (4.5,0.5) {};
        \node[fill=black,circle,inner sep=1pt]  at (1.5,0.5) {};
        \node[fill=black,circle,inner sep=1pt]  at (-1.5,0.5) {};
        \node[fill=black,circle,inner sep=1pt]  at (1.5,3.5) {};
        \node[fill=black,circle,inner sep=1pt]  at (4,1) {};
        \node at (1.5,-1) {$Z_8=C_2\oplus Z_6$};            
    \end{tikzpicture}
      \qquad 
    \begin{tikzpicture}[scale=0.6]
        \draw (-3,0)--(1,4)--(1.5,3.5)--(-2,0) -- (-1.5,0.5) -- (-1,0) -- (-1.5,0.5) -- (1.5,3.5) -- (2.5,2.5) -- (0,0) -- (2.5,2.5) -- (3,2) -- (1,0) -- (1.5,0.5) -- (2,0) -- (1.5,0.5) -- (3,2) -- (4,1) -- (3,0) -- (4,1) -- (4.5,0.5) -- (4,0) -- (4.5,0.5) -- (5,0);
        \node[fill=black,rectangle,inner sep=2pt]  at (-3,0) {};                  
        \node[fill=black,rectangle,inner sep=2pt]  at (-2,0) {};      
        \node[fill=black,rectangle,inner sep=2pt]  at (-1,0) {};  
        \node[fill=black,rectangle,inner sep=2pt]  at (0,0) {};      
        \node[fill=black,rectangle,inner sep=2pt]  at (1,0) {};      
        \node[fill=black,rectangle,inner sep=2pt]  at (2,0) {}; 
        \node[fill=black,rectangle,inner sep=2pt]  at (3,0) {}; 
        \node[fill=black,rectangle,inner sep=2pt]  at (4,0) {}; 
        \node[fill=black,rectangle,inner sep=2pt]  at (5,0) {};   
          \node[fill=black,circle,inner sep=1pt]  at (1,4) {};         
        \node[fill=black,circle,inner sep=1pt]  at (2.5,2.5) {};
        \node[fill=black,circle,inner sep=1pt]  at (3,2) {};
        \node[fill=black,circle,inner sep=1pt]  at (4.5,0.5) {};
        \node[fill=black,circle,inner sep=1pt]  at (1.5,0.5) {};
        \node[fill=black,circle,inner sep=1pt]  at (-1.5,0.5) {};
        \node[fill=black,circle,inner sep=1pt]  at (1.5,3.5) {};
        \node[fill=black,circle,inner sep=1pt]  at (4,1) {};
          \node at (1,-1) {$Z_9=C_1\oplus Z_8$};          
    \end{tikzpicture}
  
    \caption{The trees $Z_4$, $Z_5$, $Z_6$, $Z_7$, $Z_8$ and $Z_9$.}
    \label{fig:maxdecksize}
\end{figure}

See Figure~\ref{fig:maxdecksize} for small examples of the trees $Z_n$. We show that the trees $Z_n$ achieve $\max_{T:|T|=n} |\D(T)|$ and identify this maximum value.

\begin{theorem}\label{thm:maxdecksize}
For $n \ge 1$ let
$$ g(n)=\begin{cases}
n-1, & \text{if } n\in\{1,2\} \\

2\lfloor\frac{n}{3}\rfloor-1, & \text{if } n\ge 3, \,  n\equiv 0,1 \ (\textup{mod}\ 3) \\

2\lfloor\frac{n}{3}\rfloor, & \text{if } n\ge 3, \,  n\equiv 2 \ (\textup{mod}\ 3)
\end{cases} $$

Then we have the following: 

\begin{enumerate}[label={(\arabic*)}]
\item\label{case:UB} $|\mathcal{D}(T)|\le g(|T|)$ for any tree $T$.
\item\label{case:LB}  For any $n$, we have $|Z_n|=n$ and $|\mathcal{D}(Z_n)|=g(n)$.
\item\label{case:restrict} If $T$ is a tree with $|\mathcal{D}(T)|=g(|T|)$ and $3\mid n$, then $T= C_1\oplus T'$ for some tree $T'$.
\end{enumerate}
In particular, $g(n)$ is the maximum size of the deck of a size-$n$ tree.
\end{theorem}

\begin{proof}
We will prove \ref{case:UB}--\ref{case:restrict} by induction on $n$. Note that $g(n)$ is increasing. 
Using Figure~\ref{fig:smalltrees}, one can easily verify \ref{case:UB}--\ref{case:restrict} for $n \le 5$. 

Assume $n \ge 6$ and that \ref{case:UB}--\ref{case:restrict} hold for all trees of size smaller than $n$. We will consider possibilities based on the value of $n \mo{3}$. To this end, let $n = 3k+r$ where $k = \lfloor\frac{n}{3}\rfloor$ and $r\in\{0,1,2\}$. For ease of reference, we note that 
\[ g(3k) = g(3k+1) = 2k-1 \qquad \qquad g(3k+2) = 2k. \]

Let $T$ be a tree of size $n$. We will consider possible structures of $T$. For each, we compare a bound on the size of $\D(T)$ to $g(|T|)$. We will then be able to complete the inductive step for the different cases of $n \mo{3}$.

First, suppose $T =C_n$. Then $|\D(T)| = 1 = g(6)-2<g(n)$.

Next consider $T = C_2 \oplus T_2$, where $|T_2| = n-2 \ge 4$. Then by Lemma~\ref{lm:largedecomp} $|\mathcal{D}(T)|=1+|\mathcal{D}(T_2)|$. By the inductive hypothesis (1) applied to $T_2$, we have $|\D(T)| \le 1 + g(n-2)$ and that $|\D(T)| = 1+g(n-2)$ only if $|\D(T_2)| = g(n-2)$. Table~\ref{tab:maxdecksizetablec2} gives the cases for $|T|$. 

\begin{table}[h!]
    \centering
    \begin{tabular}{|c|c|c|c|}
        \hline 
        $|T|$ &  $|T_2|$ & $1+g(|T_2|)$ & $\le g(n)$? \\ 
         \hline 
        $3k$ &  $3(k-1)+1$ & $2k-2$ & $<$ \\ \hline
          $3k+1$ &  $3(k-1)+2$ & $2k-1$ & $=$ \\ \hline 
         $3k+2$ &  $3k$ & $2k$ & $=$ \\ \hline 
      \end{tabular}
    \caption{The cases of Theorem \ref{thm:maxdecksize} when $T= C_2\oplus T_2$.} 
    \label{tab:maxdecksizetablec2}
\end{table}

Next consider $T=T_1 \oplus T_2$, where $|T_1|,|T_2| \ge 3$. If $T_1= T_2$ then by Lemma~\ref{lm:largedecomp}, $|\mathcal{D}(T)|= |\mathcal{D}(T_1)|\le g( n/2)<g(n)-2$. If instead $T_1 \neq T_2$, then by Lemma~\ref{lm:largedecomp}, $|\mathcal{D}(T)|= |\mathcal{D}(T_1)|+|\mathcal{D}(T_2)|$. By the inductive hypothesis (1), $|\D(T)| \le g(|T_1|)+g(|T_2|)$.  Table~\ref{tab:maxdecksizetabletwolarge} gives the cases, up to symmetry, for $|T|, |T_1|, |T_2|$. 

\begin{table}[h!]
    \centering
        \begin{tabular}{|c|c|c|c|c|}
        \hline 
        $|T|$ & $|T_1|$ & $|T_2|$ & $g(|T_1|)+g(|T_2|)$ & $\le g(n)$? \\ \hline 
        $3k$ & $3j$ & $3(k-j)$ & $2k-2$ & $<$\\ \hline
        $3k$ & $3j+1$ & $3(k-j-1)+2$ & $2k-3$ & $<$\\ \hline
        $3k+1$ & $3j$ & $3(k-j)+1$ & $2k-2$  & $<$ \\ \hline 
        $3k+1$ & $3j+2$ & $3(k-j-1)+2$ & $2k-2$  & $<$ \\ \hline 
        $3k+2$ & $3j$ & $3(k-j)+2$ & $2k-1$  & $<$\\ \hline 
        $3k+2$ & $3j+1$ & $3(k-j)+1$ & $2k-2$  & $<$ \\ \hline 
    \end{tabular}
    \caption{The cases of Theorem \ref{thm:maxdecksize} when $T = T_1 \oplus T_2$, $|T_1|,|T_2|\ge 3$ and $T_1 \neq T_2$.} 
    \label{tab:maxdecksizetabletwolarge}
\end{table}

Finally, suppose there is an $i \ge 1$ such that $T = (C_1\oplus)^iT_2$ where $T_2\ne C_1\oplus T_3$ for any tree $T_3$. Then by Lemma~\ref{lm:smalldecomp} $|\D(T)| = 1 + |\D(T_2)|$. Since $|T_2| \ge 4$, by the inductive hypothesis (1), $|\D(T)| \le 1 + g(|T_2|)$ and $|\D(T)| = 1+g(|T_2|)$ only if $|\D(T_2)| = g(|T_2|)$. When considering cases, if $i = 1$, we have $|T_2| = n-1$ and if $i \ge 2$, we have $|T_2| \le n-2$. Table~\ref{tab:maxdecksizetablec1i} gives the relevant cases.

  \begin{table}[h!]
    \centering
    \begin{tabular}{|c|c|c|c|c|}
    \hline 
    $|T|$ & $i$ & $|T_2|$ & $1+g(|T_2|)$ & $\le g(n)$? \\ \hline
    $3k$ & $1$ & $3(k-1)+2$ & $2k-1$ & $=$ \\ \hline
    $3k$ & $\ge 2$ & $\le 3(k-1)+1$ & $\le 2k-2$ & $<$\\ \hline
    $3k+1$ & $1$ & $\star$ & $\star$& $\star$ \\ \hline 
    $3k+1$ & $\ge 2$ & $\le 3(k-1)+2$ & $\le 2k-1$ & $\le$ \\ \hline 
    $3k+2$ & $\ge 1$ & $\le 3k+1$ & $\le 2k$ & $\le$ \\ \hline 
    \end{tabular}
    \caption{The cases of Theorem \ref{thm:maxdecksize} when $T=(C_1\oplus)^iT_2$ for some $i\ge 1$ and $T_2\neq C_1\oplus T'$.} 
    \label{tab:maxdecksizetablec1i}
    \end{table}     

To complete our inductive step we consider cases based on the value of $n \mo{3}$. When $n = 3k$, we have (1). We get (2) from line one of Table~\ref{tab:maxdecksizetablec1i} and the inductive hypothesis for $Z_{n-1}$. Finally, we have $|\D(T)| < g(|T|)$ unless $T = C_1 \oplus T_2$, giving (\ref{case:restrict}).

When $n = 3k+1$, we can conclude $|\D(T)| \le g(n)$ except possibly when $T = C_1 \oplus T_2$ with $T_2 \neq C_1 \oplus T_3$ for any tree $T_3$ and with $|\D(T_2)| = g(3k)$. However, by \ref{case:restrict} such a tree does not exist, therefore \ref{case:UB} holds for $n=3k+1$. We get \ref{case:LB} from the line two of  Table~\ref{tab:maxdecksizetablec2} and the inductive hypothesis for $Z_{n-2}$.

When $n = 3k+2$, we have \ref{case:UB}. The induction hypothesis on $Z_{n-2}$ and the last line of Table~\ref{tab:maxdecksizetablec2} gives \ref{case:LB}.
\end{proof}

We now consider the maximum number of a subtrees of a tree of size $n$. Recall that, for a tree $T$, $S(T)$ is the number of subtrees $\sum_{i=1}^{|T|} |\D_i(T)|$ of $T$. For a natural number $n$, $S(n)$ is the maximum number of subtrees for a tree of size $n$, that is $S(n) = \max_{T:|T|=n} S(T)$. 

For any tree $T$, we must have that $|\D_i(T)| \le \min \left\{   \binom{n}{i}, W_i \right\}$. It follows that that $S(n) \le 2^n$. We improve this trivial upper bound in Theorem~\ref{thm:Snupperbound}. Recall (see Equation~\eqref{eqn:Wnboundsb}) that $W_i \sim \alpha \frac{b^i}{i^{3/2}}$, with $b \approx 2.4832$. 

\begin{theorem} \label{thm:Snupperbound}
The maximum number of non-isomorphic subtrees $S(n)$ is $O(1.8648^n)$.
\end{theorem}

\begin{proof}

Let $T$ be a tree of size $n$. Then since $|\D_i(T)| \le \min\left\{   \binom{n}{i}, W_i \right\}$, using Lemma~\ref{lem:Wkbounds} we have 
\[ S(T) \le \sum_{i=1}^{n} \min\Bigl\{ \binom{n}{i}, W_i          \Bigl\}\leq \sum_{i=1}^{ \lceil 0.6851n\rceil } W_i +\sum_{i= \lfloor 0.685n\rfloor }^n  \binom{n}{i}. \]
By Equation~\eqref{eqn:Wnboundsb} and the sum of a geometric progression, the first sum is in $O(b^{0.6851n})$, which is in $O(1.8648^n)$. The second sum is at most $n\cdot \binom{n}{\lfloor 0.685n \rfloor }$. Recall from the proof of Lemma~\ref{lem:Wkbounds} that $\binom{n}{cn}\leq (c^{-c}(1-c)^{c-1})^n$. Hence the second sum is in $O(1.8646^n)$. 
\end{proof}

For a lower bound on $S(n)$, Dossou-Olory~\cite{fibonacci_trees} shows that $S(n) \in \Omega (1.40^n)$.
We improve this lower bound in Theorem~\ref{thm:Snlowerbound} using the following family of trees.

\begin{definition}
For $n \ge 1$ with $n \equiv 1 \mo{4}$, we recursively define trees $X_n$. Let $X_1 = C_1$. For $n \ge 5$, let $X_n =C_3 \oplus (C_1 \oplus X_{n-4})$. For $n \ge 2$ with $n \equiv 2 \mo{4}$, we define trees $Y_n$ by $Y_n = C_1 \oplus X_{n-1}$. 
\end{definition}

Observe that $X_n = C_3\oplus Y_{n-3}$. See Figure~\ref{fig:XYtrees} for the two representations of $X_n$. 

\begin{figure}
\centering
\scalebox{0.6}{
    \begin{tikzpicture}
        \pic at (0,0) {tree3};
        \node[fill=black,rectangle,inner sep=2pt]  at (3,0) {};  
        \draw (6,2) -- (4,0) -- (8,0) -- (6,2) -- (4,4) -- (1,1);
        \draw (3,0) -- (5.5,2.5);
        \node at (6,0.5) {\Large{$X_{n-4}$}};

        \pic at (10,0) {tree3};
        \draw (15.5,2.5) -- (13,0) -- (18,0) -- (15.5,2.5) -- (14,4) -- (11,1);
        \node at (15.5,0.5) {\Large{$Y_{n-3}$}};
    \end{tikzpicture}
}
\caption{Two representations of $X_n$ for $n \equiv 1 \mo{4}$. 
} \label{fig:XYtrees}
\end{figure}
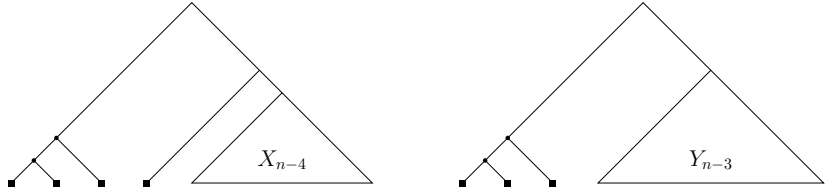

\begin{lemma} \label{lem:XandYsubtrees}
For any $n\geq 5$ with $n \equiv 1 \mo{4}$, the following recurrences hold:
\begin{equation}\label{eqn:recurrence}
\begin{split}
S(Y_{n+5}) & = 2\cdot S(X_{n+4})-S(Y_{n+1}),\\
S(X_{n+8}) & =4\cdot S(Y_{n+5})-S(X_{n+4})-2\cdot S(Y_{n+1}).
\end{split} 
\end{equation}
\end{lemma}

\begin{proof}
We start with first recurrence for $Y_{n+5} = C_1\oplus X_{n+4}$. Let $\mathcal{A}_1$ be the set of subtrees of $Y_{n+5}$ of the form $C_1 \oplus T$ for some subtree $T$ of $X_{n+4}$ and let $\mathcal{A}_2$ be the set of subtrees of $Y_{n+5}$ not of this form. Then each tree in $\mathcal{A}_1$ and $\mathcal{A}_2$ is determined up to isomorphism by a subtree $T$ of $X_{n+4}$. Hence, by the inclusion-exclusion principle, 
\[S(Y_{n+5})=|\mathcal{A}_1|+|\mathcal{A}_2|-|\mathcal{A}_1\cap \mathcal{A}_2|=2\cdot S(X_{n+4})-|\mathcal{A}_1\cap \mathcal{A}_2|.\]
We claim $|\mathcal{A}_1 \cap \mathcal{A}_2| = S(Y_{n+1})$. We have that $\mathcal{A}_1 \cap \mathcal{A}_2$ consists of the subtrees of $X_{n+4}$ of the form $C_1\oplus T$ for another subtree $T$ of $X_{n+4}$. Since $X_{n+4}= C_3\oplus Y_{n+1}$, we can obtain a bijection between $\mathcal{A}_1 \cap \mathcal{A}_2$ and subtrees $T$ of $Y_{n+1}$ by mapping $C_1\oplus T$ to $T$. Thus, the first recurrence of Equation~(\ref{eqn:recurrence}) holds. 

For the second recurrence, we consider $X_{n+8} = C_3\oplus Y_{n+5}$. We define four subsets $\mathcal{A}_0, \mathcal{A}_1, \mathcal{A}_2, \mathcal{A}_3$ of subtrees of $X_{n+8}$ where trees in $\mathcal{A}_i$ have $i$ leaves from the $C_3$. More precisely, let $\mathcal{A}_0$ be the subtrees of $X_{n+8}$ that are subtrees of $Y_{n+5}$ and for $1 \le i \le 3$, let $\mathcal{A}_i$ be the subtrees of $X_{n+8}$ of that are of the form $C_i \oplus T$ for some subtree $T$ of $Y_{n+5}$. Then for $0 \le i \le 3$, $|\mathcal{A}_i| = S(Y_{n+5})$. 

Since $n \ge 5$, for distinct $i,j \in \{1,2,3\}$ we have $\mathcal{A}_i \cap \mathcal{A}_j = \{C_i \oplus C_j\}$. Hence,  
\[|\mathcal{A}_1\cap \mathcal{A}_2|=|\mathcal{A}_1\cap \mathcal{A}_3|=|\mathcal{A}_2\cap \mathcal{A}_3|=|\mathcal{A}_0\cap \mathcal{A}_1\cap \mathcal{A}_2|=|\mathcal{A}_0\cap \mathcal{A}_1\cap \mathcal{A}_3|=|\mathcal{A}_0\cap \mathcal{A}_2\cap \mathcal{A}_3|=1\]
and 
\[|\mathcal{A}_1\cap \mathcal{A}_2\cap \mathcal{A}_3|=|\mathcal{A}_0\cap \mathcal{A}_1\cap \mathcal{A}_2\cap \mathcal{A}_3|=0.\]
Hence, by the inclusion-exclusion principle, 
\[S(X_{n+8})=4\cdot S(Y_{n+5})-|\mathcal{A}_0\cap \mathcal{A}_1|-|\mathcal{A}_0\cap \mathcal{A}_2|-|\mathcal{A}_0\cap \mathcal{A}_3|.\]

We have that $\mathcal{A}_0\cap \mathcal{A}_1$ consists of subtrees of $Y_{n+5}$ that are of the form $C_1 \oplus T$ for another subtree $T$ of $Y_{n+5}$. Since $Y_{n+5}= C_1\oplus X_{n+4}$, we obtain a bjection between $\mathcal{A}_0 \cap \mathcal{A}_1$ and subtrees $T$ of $X_{n+4}$ by mapping $C_1 \oplus T$ to $T$. Next, we have that $\mathcal{A}_0 \cap \mathcal{A}_2$ consists of subtrees of $Y_{n+5}$ of the form $C_2 \oplus T$ for another subtree $T$ of $Y_{n+5}$. Since $Y_{n+5} = C_1 \oplus (C_3 \oplus Y_{n+1})$, we can obtain a bijection between $\mathcal{A}_0 \cap \mathcal{A}_2$ and subtrees of $Y_{n+1}$. We obtain a similar bijection between $\mathcal{A}_0 \cap \mathcal{A}_3$ and subtrees of $Y_{n+1}$. Hence the second recurrence of Equation~(\ref{eqn:recurrence}) holds. 
\end{proof}

\begin{theorem} \label{thm:Snlowerbound}
The maximum number of non-isomorphic subtrees $S(n)$ is $\Omega(5^{n/4})$.
\end{theorem}

\begin{proof}
We show that $S(Y_m) =  1 + 8 \cdot 5^{(m-6)/4}$ for $m \equiv 2 \mo{4}$. 

Let $n \equiv 1 \mo{4}$. Solving the recurrence for $S(Y_{n+5})$ in Equation~\eqref{eqn:recurrence} in terms of $S(X_{n+4})$ yields 
\[ S(X_{n+4}) = \frac{1}{2} ( S(Y_{n+5}) + S(Y_{n+1})). \]
Using this recurrence and substituting for $S(X_{n+4})$ and $S(X_{n+8})$ in the second line of Equation~\eqref{eqn:recurrence} yields
\[ S(Y_{n+9}) - 6 S(Y_{n+5})+5S(Y_{n+1}) = 0. \] 

Rewriting for $m \equiv 2 \mo{4}$, we let $m = 4k+2$, obtaining
\[ S(Y_{4(k+2)+2}) - 6 S(Y_{4(k+1)+2})+5S(Y_{4k+2}) = 0.
\]
Solving the related characteristic equation we find the solution for $S(Y_{4k+2})$ is of the form $a \cdot 1^k + b \cdot 5^k$ for constants $a$ and $b$. 
As base cases, we can computationally verify that $S(Y_6) = 9$ and $S(Y_{10}) = 41$. Hence $a = 1$ and $b = \frac{8}{5}$. We conclude $S(Y_{4k+2}) = 1 + 8 \cdot 5^{k-1}$ so that 
\[ S(Y_m) = 1 + 8 \cdot 5^{(m-6)/4}.\]
It follows that $S(n)  \in \Omega(5^{n/4})$. 
\end{proof}

We note that since $S(Y_n) \in \Theta(5^{n/4})$ for $n \equiv 2 \mo{4}$, using Equation~\eqref{eqn:recurrence} we also have $S(X_n) \in \Theta(5^{n/4})$ for $n \equiv 1 \mo{4}$.

\subsection{Universal trees} \label{subsec:universal}
We now consider rooted binary universal trees. 

\begin{definition}
A rooted binary tree $T$ is called \emph{$k$-universal} if $\vert \mathcal{D}(T) \vert = W_k$, i.e., if its size-$k$ deck contains all size-$k$ trees. Moreover, $u(k)$ denotes the minimum number of leaves in a $k$-universal tree.
\end{definition}

Upper and lower bounds for the minimum size of $k$-universal trees have recently been obtained by Gawrychowski et al.~\cite{Gawrychowski2018}, albeit using slightly different terminology. More precisely,~\cite{Gawrychowski2018} considered the notion of minor-universal trees, where a rooted tree $T$ is called a \emph{minor-universal tree} for a family of rooted trees, $\mathcal{S}$ say, if every tree $S \in \mathcal{S}$ is a topological minor of $T$, meaning that $T$ contains a subdivision of $S$ as a subgraph. Since every leaf-induced subtree of a rooted binary tree $T$ is a topological minor of $T$, a consequence of the results in~\cite{Gawrychowski2018} (in particular, Theorem 4 and Theorem 11 therein) is the following corollary.

\begin{corollary}
    The following hold: 
    \begin{enumerate}
        \item $u(k) \in O(k^{1.894})$, and
        \item $u(k) \in \Omega(k^{1.728})$.
    \end{enumerate}
\end{corollary}

To complement these asymptotics, we explicitly calculated $u(k)$ for small $k$ and list the minimum size $k$-universal trees for $k \in [11]$ in Appendix~\ref{sec:computation}. Additionally, we depict a $12$-universal tree on 28 leaves in Appendix~\ref{sec:12universal}. The latter illustrates that $u(12) \leq 28$ (see Table~\ref{tab:kalmar}).

To our surprise, the computed values of $u(k)$ coincided for a while with the terms of the sequence A173382 from \cite{oeis}, and even when they differed, they remained close. Sequence A173382 is defined as the  partial sums of sequence A074206, which counts the ordered factorizations of the number $k$. We have no explanation for this phenomenon, which might be a pure coincidence. The table below compares the two sequences. The sequence A173382  was studied by L. Kalm\'ar \cite{kalmarnemet,kalmarmagyar}, who showed that the sequence is asymptotically $c\cdot k^s$, where $s>1$ is the single real root of $\zeta(s)=2$. Numerically, $s\approx 1.72864723$.

\begin{table}[htbp]
\centering
\begin{tabular}{|c|c|c|c|c|c|c|c|c|c|c|c|c|}\hline
	& $ 1 $ & $2$ & $3$ & $4$ & $5$ & $6$ & $7$ & $8$ & $9$ & $10$ & $11$ & $12$\\\hline
A173382 & $ 1 $ & $2$ & $3$ & $5$ & $6$ & $9$ & $10$ & $14$ & $16$ & $19$ & $20$ & $28$  \\\hline
$u(k)$ & $ 1 $ & $2$ & $3$ & $5$ & $6$ & $9$ & $10$ & $14$ & $16$ & $19$ & $21$ & $\leq 28$ \\\hline
\end{tabular}
\caption{$u(k)$ versus the Kalm\'ar sequence A173382.}\label{tab:kalmar}
\end{table}

\section{Open Problems} \label{sec:openproblems}
In this paper, we have considered problems related to decks and multidecks of rooted binary trees. We have focused on questions related to reconstruction and minimum (resp. maximum) cardinalities. We now conclude by pointing out some open problems and directions for future research. 

For reconstruction, we showed (i) $R(n) \geq 2 \lceil \frac{n}{4} \rceil$ for $n \geq 6$ (Theorem~\ref{decklowerbound}), (ii) $R^{(m)}(n)\geq (1-o(1))\frac{\log n}{\log b}$ with $b$ as in Equation~\eqref{eqn:Wnboundsb} (Theorem~\ref{thm:multilow}), and for $n \notin \{4,5\}$, we have $R(n), R^{(m)}(n) \leq n-1$ (Theorem~\ref{th:deck} and Corollary~\ref{th:multi}). An interesting open problem would be to strengthen these bounds if possible.

\begin{problem}
Improve the lower and upper bounds for $R(n)$ and $R^{(m)}(n)$. 
\end{problem}

In terms of deck cardinalities, we have fully characterized the trees with minimum size-$i$ decks for all $i \in [n]$. Additionally, we have obtained asymptotic upper and lower bounds for $S(n)$, the maximum number of non-isomorphic subtrees of size-$n$ trees. In particular, $S(n) \in O(1.8648^n)$ (Theorem~\ref{thm:Snupperbound}) and $S(n) \in \Omega(5^{n/4})$ (Theorem~\ref{thm:Snlowerbound}). An immediate question that arises is whether these bounds can be improved.

\begin{problem}
Improve the asymptotic lower and upper bounds for $S(n)$. 
\end{problem}

Finally, recall that a tree is $k$-universal if $\vert \D_k(T) \vert = W_k$. Trivially, a $k$-universal tree thus has a maximum size-$k$ deck. Here, we have focused on determining $u(k)$, i.e., the minimum leaf number required for a tree to be $k$-universal. A related but different question is to fix a leaf number $n$ and analyze the maximum cardinality of the size-$j$ deck for each $j < n-1$. 
\begin{problem}
Fix a positive integer $n$. 
For each  $j<n-1$ such that $u(j) > n$, determine $\max_{|T|=n}\{|\mathcal{D}_j(T)|\}$. Which trees achieve this maximum?
\end{problem}

\section*{Acknowledgements}

This material is based upon work supported by the National Science Foundation under Grant Number DMS 1641020.

\bibliographystyle{plain}
\bibliography{bibliography}

\begin{thebibliography}{10}

\bibitem{Bondy1977}
J.~A. Bondy and R.~L. Hemminger.
\newblock Graph reconstruction{\textemdash}a survey.
\newblock {\em Journal of Graph Theory}, 1(3):227--268, 1977.

\bibitem{Cover-Thomas}
Thomas~M. Cover and Joy~A. Thomas.
\newblock {\em Elements of Information Theory (Wiley Series in
  Telecommunications and Signal Processing)}.
\newblock Wiley-Interscience, USA, 2006.

\bibitem{fibonacci_trees}
Audace Amen~Vioutou Dossou-Olory.
\newblock Leaf-induced subtrees of leaf-{F}ibonacci trees.
\newblock {\em Discrete Mathematics Letters}, 1:1--7, 2019.

\bibitem{Gawrychowski2018}
Pawe{\l} Gawrychowski, Fabian Kuhn, Jakub {\L}opusza{\'{n}}ski, Konstantinos
  Panagiotou, and Pascal Su.
\newblock Labeling schemes for nearest common ancestors through minor-universal
  trees.
\newblock In {\em Proceedings of the Twenty-Ninth Annual {ACM}-{SIAM} Symposium
  on Discrete Algorithms}, pages 2604--2619. Society for Industrial and Applied
  Mathematics, January 2018.

\bibitem{Harary1974}
Frank Harary.
\newblock A survey of the reconstruction conjecture.
\newblock In {\em Lecture Notes in Mathematics}, pages 18--28. Springer Berlin
  Heidelberg, 1974.

\bibitem{kalmarnemet}
L.~Kalm\'ar.
\newblock \"{U}ber die mittlere {A}nzahl der {P}roduktdarstellungen der
  {Z}ahlen, ({E}rste {M}itteilung).
\newblock {\em Acta Litterarum ac Scientiarum Regiae Universitatis Hungaricae
  Francisco-Josephinae sectio Scientiarum Mathematicarum}, pages 95--107, 1930.

\bibitem{kalmarmagyar}
L.~Kalm\'ar.
\newblock A "factorisatio numerorum" probl\'em\'aj\'ar\'ol.
\newblock {\em Matematikai \'es Fizikai Lapok}, XXXVIII:1--15, 1931.

\bibitem{Kelly1942}
P.~J. Kelly.
\newblock {\em On isometric transformations}.
\newblock PhD thesis, University of Wisconsin-Madison, 1942.

\bibitem{oeis}
{OEIS Foundation Inc.}
\newblock The on-line encyclopedia of integer sequences.
\newblock \url{ http://oeis.org}, 2022.

\bibitem{number_of_trees}
R.~Otter.
\newblock The number of trees.
\newblock {\em Annals of Mathematics}, 49(3):583--599, 1948.

\bibitem{Ramachandran2004}
S.~Ramachandran and S.~Arumugam.
\newblock Graph reconstruction - some new developments*.
\newblock {\em AKCE International Journal of Graphs and Combinatorics},
  1(1):51--61, 2004.

\bibitem{SempleSteel}
Charles Semple and Mike Steel.
\newblock {\em Phylogenetics ({O}xford {L}ecture {S}eries in {M}athematics and
  {I}ts {A}pplications)}.
\newblock Oxford University Press, 2003.

\bibitem{Ulam1960}
S.~M. Ulam.
\newblock {\em A collection of mathematical problems}.
\newblock Interscience Publishers New York, 1960.

\end{thebibliography}

\appendix
\section{Minimum size \texorpdfstring{$k$}{k}-universal trees for small \texorpdfstring{$k$}{k}} \label{sec:computation}

For $k\in\{1,2,3\}$, the rooted binary tree with $k$ leaves is unique, and $u(k)=k$. In the following pictures, $U_k$ will denote the $k$-universal tree of minimal size (if it is unique) and
$U_k^i$ will denote the $k$-universal trees of minimal size, where $i$ ranges between 1 and the total number of $k$-universal trees of minimal size.

\begin{figure}[H]
\centering
\begin{tikzpicture}[scale=.5,font=\tiny]    
   \node[whiteroot]  at (0,0) {};      
\node at (0,-1) {$U_1$};  
\end{tikzpicture}\quad\quad
\begin{tikzpicture}[scale=.5,font=\tiny]    
\draw (0,0) -- (0.5,0.5) -- (1,0);
   \node[wleaf]  at (0,0) {};      
   \node[wleaf]  at (1,0) {};      
   \node[rootvertex]  at (0.5,0.5) {};
\node at (.5,-1) {$U_2$};  
\end{tikzpicture}\quad\quad
\begin{tikzpicture}[scale=.5,font=\tiny]    
\draw (0,0) -- (0.5,0.5) -- (1,0);
   \draw (0.5,0.5) -- (1,1) -- (2,0);
   \node[wleaf]  at (0,0) {};      
   \node[wleaf]  at (1,0) {};      
   \node[wleaf]  at (2,0) {}; 
   \node[invertex]  at (0.5,0.5) {};
   \node[rootvertex]  at (1,1) {};
\node at (1,-1) {$U_3$};  
\end{tikzpicture}\quad\quad
\begin{tikzpicture}[scale=.5,font=\tiny]    
\draw (0,0) -- (0.5,0.5) -- (1,0);
   \draw (0.5,0.5) -- (1,1) -- (2,0);
   \draw (3,0) -- (3.5,0.5) -- (4,0);
   \draw (1,1) -- (2,2) -- (4,0);
   \node[wleaf]  at (0,0) {};      
   \node[wleaf]  at (1,0) {};      
   \node[wleaf]  at (2,0) {}; 
   \node[wleaf]  at (3,0) {}; 
   \node[wleaf]  at (4,0) {}; 
   \node[invertex]  at (0.5,0.5) {};
   \node[invertex]  at (1,1) {};
   \node[rootvertex]  at (2,2) {};
   \node[invertex]  at (3.5,0.5) {};
\node at (2,-1) {$U_4^1$};  
\end{tikzpicture}\quad\quad
\begin{tikzpicture}[scale=.5,font=\tiny]    
\draw (0,0) -- (0.5,0.5) -- (1,0);
   \draw (2.5,0.5) -- (2,0);
   \draw (1.5,1.5) -- (3,0);
   \draw (0.5,0.5) -- (2,2) -- (4,0);
   \node[wleaf]  at (0,0) {};      
   \node[wleaf]  at (1,0) {};      
   \node[wleaf]  at (2,0) {}; 
   \node[wleaf]  at (3,0) {}; 
   \node[wleaf]  at (4,0) {}; 
   \node[invertex]  at (0.5,0.5) {};
   \node[rootvertex]  at (2,2) {};
   \node[invertex]  at (1.5,1.5) {};    
   \node[fill=black,circle,inner sep=1pt]  at (2.5,0.5) {}; 
\node at (2,-1) {$U_4^2$};  
\end{tikzpicture}\quad\quad
\begin{tikzpicture}[scale=.5,font=\tiny]    
   \draw (1,1) -- (2,2) -- (4,0);
   \draw (1,1)--(0,0)--(2,0)--cycle;
   \node at (1,.35) {$U_4^1$};
   \node[wleaf]  at (4,0) {}; 
   \node[rootvertex]  at (2,2) {}; 
\node at (2,-1) {$U_5$};  
\end{tikzpicture}
\caption{The minimum size $k$-universal trees for $k\in\{1,2,3,4,5\}$.} 
\end{figure}

\begin{figure}[H]
\centering
\begin{tikzpicture}[scale=.5,font=\tiny] 
   \draw (-.1,-.1)--(2.1,-.1)--(1,1)--cycle;
   \node at (1,0.35) {$C_3$};  
   \draw (2.35,-.1)--(4.45,-.1)--(3.35,1)--cycle;
   \node at (3.35,0.35) {$C_3$};  
   \draw (1,1)--(3.25,3.25);
   \draw(2.175,2.175)--(3.35,1);
   \draw(2.8375,2.8375)--(5.275,0.5);
   \draw(4.75,0)--(5.25,0.5)--(5.75,0);
   \draw(3.25,3.25)--(6.5,0);
   \node[wleaf]  at (4.8,0) {};   
   \node[wleaf]  at (5.8,0) {}; 
   \node[wleaf]  at (6.5,0) {}; 
   \node[invertex]  at (2.175,2.175) {};
   \node[invertex]  at (2.8375,2.8375) {};
   \node[rootvertex]  at (3.25,3.25) {};
   \node[invertex]  at (5.275,0.5) {};
\node at (3.25,-1) {$U_6^1$};  
\end{tikzpicture}
\quad
\begin{tikzpicture}[scale=.5,font=\tiny]   
\draw (-.1,-.1)--(2.1,-.1)--(1,1)--cycle;
   \node at (1,0.35) {$C_3$};  
   \draw (2.35,-.1)--(4.35,-.1)--(3.35,1)--cycle;
   \node at (3.35,0.35) {$C_3$};  
   \draw (1,1)--(3.25,3.25);
   \draw(2.175,2.175)--(3.35,1);
   \draw(2.425,2.425)--(4.85,0);
   \draw(5.5,0)--(6,0.5)--(6.5,0);
   \draw(3.25,3.25)--(6,0.5);
   \node[wleaf]  at (4.85,0) {};   
   \node[wleaf]  at (5.5,0) {}; 
   \node[wleaf]  at (6.5,0) {}; 
   \node[invertex]  at (2.175,2.175) {};
   \node[invertex]  at (2.425,2.425) {};
   \node[rootvertex]  at (3.25,3.25) {};
   \node[invertex]  at (6,0.5) {};
   \node at (3.25,-1) {$U_6^2$};  
   \end{tikzpicture}
\quad
\begin{tikzpicture}[scale=.5, font=\tiny]  
\draw (1,1) -- (2.5,2.5) -- (4,1);
 \draw (2.5,2.5) -- (3,3) -- (6,0);
 \draw (1,1)--(-0.1,-.1)--(2.1,-.1)--cycle;
 \node at (1,0.35) {$U_4^1$};  
 \draw(2.9,-.1)--(4,1)--(5.1,-.1)--cycle;
 \node at (4,0.35) {$C_3$};  
 \node[wleaf]  at (6,0) {};
 \node[invertex]  at (2.5,2.5) {};
 \node[rootvertex]  at (3,3) {};
\node at (3,-1) {$U_6^3$};  \end{tikzpicture}
\quad
\begin{tikzpicture}[scale=.5,font=\tiny] 
 \draw (1,1) -- (2.5,2.5) -- (4,1);
 \draw (2.5,2.5) -- (3,3) -- (6,0);
 \draw (1,1)--(-0.1,-.1)--(2.1,-.1)--cycle;
 \node at (1,0.35) {$U_4^2$};  
 \draw(2.9,-.1)--(4,1)--(5.1,-.1)--cycle;
 \node at (4,0.35) {$C_3$};  
 \node[wleaf]  at (6,0) {};
 \node[invertex]  at (2.5,2.5) {};
 \node[rootvertex]  at (3,3) {};
\node at (3,-1) {$U_6^4$};  \end{tikzpicture}
\quad
\begin{tikzpicture}[scale=.5,font=\tiny] 
\draw (1,1) -- (2.25,2.25) -- (3.5,1);
 \draw (2.25,2.25) -- (2.75,2.75) -- (5.5,0);
 \draw (1,1)--(-0.1,-.1)--(2.1,-.1)--cycle;
 \node at (1,0.35) {$C_4$};  
 \draw(2.4,-.1)--(3.5,1)--(4.6,-.1)--cycle;
 \node at (3.5,0.35) {$B_2$};  
 \node[wleaf]  at (5.5,0) {};
 \node[invertex]  at (2.25,2.25) {};
 \node[rootvertex]  at (2.75,2.75) {};
\node at (2.75,-1) {$U_6^5$}; 
\end{tikzpicture}
\quad
\begin{tikzpicture}[scale=.5,font=\tiny] 
\draw (1,1) -- (2.5,2.5) -- (4,1);
 \draw (1,1)--(-0.1,-.1)--(2.1,-.1)--cycle;
 \node at (1,0.35) {$U_5$};  
 \draw(2.9,-.1)--(4,1)--(5.1,-.1)--cycle;
 \node at (4,0.35) {$C_3$};  
 \node[rootvertex]  at (2.5,2.5) {};
\node at (2.5,-1) {$U_6^6$};  
\end{tikzpicture}
\quad
\begin{tikzpicture}[scale=.5,font=\tiny] 
 \draw (1,1) -- (2.25,2.25) -- (3.5,1);
 \draw (2.25,2.25) -- (2.75,2.72) -- (5.5,0);
 \draw (1,1)--(-0.1,-.1)--(2.1,-.1)--cycle;
 \node at (1,0.35) {$U_5$};  
 \draw(2.4,-.1)--(3.5,1)--(4.6,-.1)--cycle;
 \node at (3.5,0.35) {$C_3$};  
 \node[wleaf]  at (5.5,0) {};
 \node[invertex]  at (2.25,2.25) {};
 \node[rootvertex]  at (2.75,2.75) {};
\node at (2.75,-1) {$U_7$};  
\end{tikzpicture}
\caption{The minimum size $k$-universal trees for $k\in\{6,7\}$.} 
\end{figure}
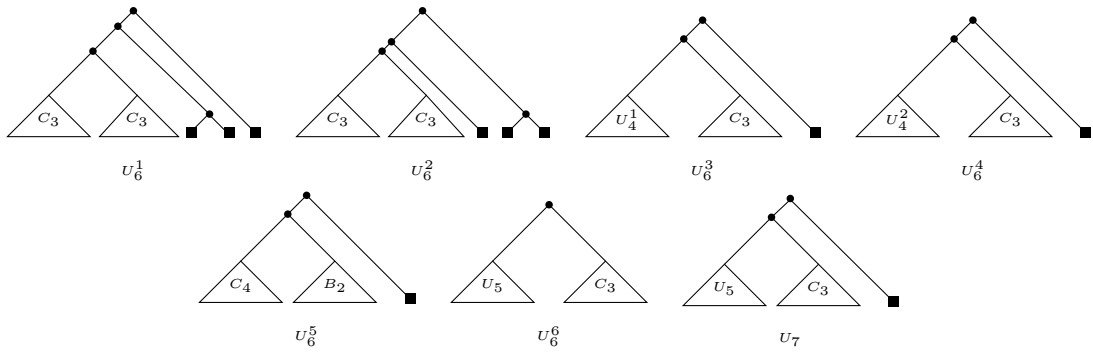

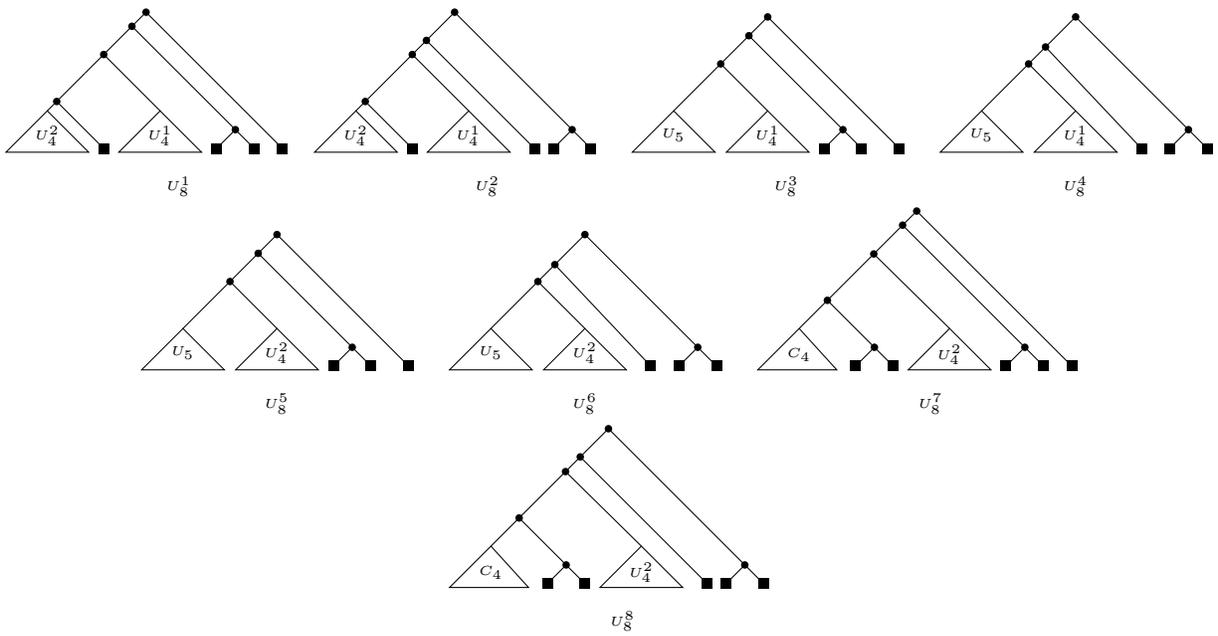
\begin{figure}[H]
\centering
\begin{tikzpicture}[scale=.5,font=\tiny,baseline={(0,0)}] 
   \draw (1.25,1.25) -- (2.5,0);
   \draw (1,1) -- (2,2) ;
   \draw (2,2) -- (2.5,2.5) -- (4,1);%
   \draw (2.5,2.5) -- (3.25,3.25) -- (6.5,0);%
   \draw (3.25,3.25) -- (3.625,3.625) -- (7.25,0);%
   \draw (5.5,0) -- (6,0.5);%
   \draw (1,1)--(-.1,-.1)--(2.1,-.1)--cycle;
   \node  at (1,0.35) {$U_4^2$};  
   \node[wleaf]  at (2.5,0) {}; 
   \node[wleaf]  at (5.5,0) {};%
   \node[wleaf]  at (6.5,0) {}; %
   \node[wleaf]  at (7.25,0) {}; %
   \draw (4,1)--(2.9,-.1)--(5.1,-.1)--cycle;    
   \node at (4,0.35) {$U_4^1$};   
   \node[invertex]  at (1.25,1.25) {};  
   \node[invertex]  at (2.5,2.5) {}; 
   \node[rootvertex]  at (3.625,3.625) {}; %
   \node[invertex]  at (3.25,3.25) {}; %
   \node[invertex]  at (6,0.5) {}; %
\node at (4.5,-1) {$U_8^{1}$};  
\end{tikzpicture}\quad
\begin{tikzpicture}[scale=.5,font=\tiny,baseline={(0,0)}] 
   \draw (1,1) -- (3.625,3.625);
   \draw (1.25,1.25) -- (2.5,0);
   \draw (2.5,2.5) -- (4,1);
   \draw (2.875,2.875) -- (5.75,0);
   \draw (3.625,3.625) -- (7.25,0);
   \draw (6.25,0) -- (6.75,0.5);
   \draw (1,1)--(-.1,-.1)--(2.1,-.1)--cycle;
   \draw (4,1)--(2.9,-.1)--(5.1,-.1)--cycle;   
   \node at (1,.35) {$U_4^2$};  
   \node[wleaf]  at (2.5,0) {}; 
   \node[wleaf]  at (5.75,0) {}; 
   \node[wleaf]  at (6.25,0) {}; 
   \node[wleaf]  at (7.25,0) {}; 
   \node at (4,0.35) {$U_4^1$};   
   \node[invertex]  at (1.25,1.25) {};  
   \node[invertex]  at (2.5,2.5) {}; 
   \node[invertex]  at (2.875,2.875) {}; 
   \node[rootvertex]  at (3.625,3.625) {}; 
   \node[invertex]  at (6.75,0.5) {}; 
\node at (4.5,-1) {$U_8^{2}$};  
\end{tikzpicture}
\quad
\begin{tikzpicture}[scale=.5,font=\tiny,baseline={(0,0)}] 
   \draw (3,3)--(3.5,3.5)--(7,0);
   \draw (5.5,0.5) -- (6,0);
   \draw (2.25,2.25) -- (3,3) -- (5.5,0.5);
   \draw (1,1) -- (2.25,2.25) -- (3.5,1);
   \draw (5,0) -- (5.5,0.5);
   \draw (1,1)--(-.1,-.1)--(2.1,-.1)--cycle;
   \draw (3.5,1) -- (2.4,-.1)-- (4.6,-.1)--cycle;
   \node at (1,0.35) {$U_5$};   
   \node[wleaf]  at (5,0) {}; 
   \node[wleaf]  at (6,0) {}; 
   \node[wleaf]  at (7,0) {}; 
   \node at (3.5,0.35) {$U_4^1$};  
   \node[invertex]  at (2.25,2.25) {};
   \node[invertex]  at (3,3) {}; 
   \node[rootvertex]  at (3.5,3.5) {}; 
   \node[invertex]  at (5.5,0.5) {}; 
\node at (4,-1) {$U_8^{3}$};  
\end{tikzpicture}
\quad
\begin{tikzpicture}[scale=.5,font=\tiny,baseline={(0,0)}] 
   \draw(2.25,2.25)--(3.5,1);
   \draw(1,1)--(2.7,2.7);
   \draw(2.7,2.7)--(5.25,0);
   \draw (2.75,2.75) -- (3.5,3.5) -- (7,0);
   \draw (6,0) -- (6.5,0.5);
   \draw (1,1)--(-.1,-.1)--(2.1,-.1)--cycle;
   \draw (3.5,1) -- (2.4,-.1)-- (4.6,-.1)--cycle;
   \node at (1,0.35) {$U_5$};   
   \node[wleaf]  at (5.25,0) {}; 
   \node[wleaf]  at (6,0) {}; 
   \node[wleaf]  at (7,0) {}; 
   \node at (3.5,0.35) {$U_4^1$};  
   \node[invertex]  at (2.25,2.25) {};
   \node[invertex]  at (2.7,2.7) {}; 
   \node[rootvertex]  at (3.5,3.5) {}; 
   \node[invertex]  at (6.5,0.5) {}; 
\node at (3.5,-1) {$U_8^{4}$};  
\end{tikzpicture}\quad
\begin{tikzpicture}[scale=.5,font=\tiny,baseline={(0,0)}] 
    \draw (3,3)--(3.5,3.5)--(7,0);
   \draw (5.5,0.5) -- (6,0);
   \draw (2.25,2.25) -- (3,3) -- (5.5,0.5);
   \draw (1,1) -- (2.25,2.25) -- (3.5,1);
   \draw (5,0) -- (5.5,0.5);
   \draw (1,1)--(-.1,-.1)--(2.1,-.1)--cycle;
   \draw (3.5,1) -- (2.4,-.1)-- (4.6,-.1)--cycle;
   \node at (1,0.35) {$U_5$};   
   \node[wleaf]  at (5,0) {}; 
   \node[wleaf]  at (6,0) {}; 
   \node[wleaf]  at (7,0) {}; 
   \node at (3.5,0.35) {$U_4^2$};  
   \node[invertex]  at (2.25,2.25) {};
   \node[invertex]  at (3,3) {}; 
   \node[rootvertex]  at (3.5,3.5) {}; 
   \node[invertex]  at (5.5,0.5) {}; 
\node at (3.5,-1) {$U_8^{5}$};  
\end{tikzpicture}
\quad
\begin{tikzpicture}[scale=.5,font=\tiny,baseline={(0,0)}] 
  \draw(2.25,2.25)--(3.5,1);
   \draw(1,1)--(2.7,2.7);
   \draw(2.7,2.7)--(5.25,0);
   \draw (2.75,2.75) -- (3.5,3.5) -- (7,0);
   \draw (6,0) -- (6.5,0.5);
   \draw (1,1)--(-.1,-.1)--(2.1,-.1)--cycle;
   \draw (3.5,1) -- (2.4,-.1)-- (4.6,-.1)--cycle;
   \node at (1,0.35) {$U_5$};   
   \node[wleaf]  at (5.25,0) {}; 
   \node[wleaf]  at (6,0) {}; 
   \node[wleaf]  at (7,0) {}; 
   \node at (3.5,0.35) {$U_4^2$};  
   \node[invertex]  at (2.25,2.25) {};
   \node[invertex]  at (2.7,2.7) {}; 
   \node[rootvertex]  at (3.5,3.5) {}; 
   \node[invertex]  at (6.5,0.5) {}; 
\node at (3.5,-1) {$U_8^{6}$}; 
\end{tikzpicture}
\quad
\begin{tikzpicture}[scale=.5,font=\tiny,baseline={(0,0)}] 
   \draw(0.9,-0.1)--(3,-0.1)--(2,1)--cycle;
   \node at (2,0.3) {$C_4$}; 
   \draw(4.9,-0.1)--(7.1,-0.1)--(6,1)--cycle;
   \node at (6,0.3) {$U_4^2$}; 
   \draw(2,1)--(4.75,3.75)--(8,0.5);
   \draw(2.75,1.75)--(4,0.5);
   \draw(3.5,0)--(4,0.5)--(4.5,0);
   \draw (3.983825,2.983825)--(6,1);
   \draw(7.5,0)--(8,0.5)--(8.5,0);
   \draw(4.75,3.75)--(5.125,4.125)--(9.25,0);
   \node[wleaf]  at (3.5,0) {}; 
   \node[wleaf]  at (4.5,0) {};
   \node[wleaf]  at (7.5,0) {};
   \node[wleaf]  at (8.5,0) {};
   \node[wleaf]  at (9.25,0) {};
   \node[invertex]  at (2.75,1.75) {};
   \node[invertex]  at (4,0.5) {};
   \node[invertex]  at (3.983825,2.983825) {};
   \node[invertex]  at (4.75,3.75) {};
    \node[rootvertex]  at (5.125,4.125) {};
   \node[invertex]  at (8,0.5) {};
   \node at (5.5,-1) {\text{ $U_8^{7}$ }};
\end{tikzpicture}
\quad
\begin{tikzpicture}[scale=.5,font=\tiny,baseline={(0,0)}] 
   \draw(0.9,-0.1)--(3,-0.1)--(2,1)--cycle;
   \node at (2,0.3) {$C_4$}; 
   \draw(4.9,-0.1)--(7.1,-0.1)--(6,1)--cycle;
   \node at (6,0.3) {$U_4^2$}; 
   \draw(2,1)--(4.375,3.375)--(7.75,0);
   \draw(2.75,1.75)--(4,0.5);
   \draw(3.5,0)--(4,0.5)--(4.5,0);
   \draw (3.983825,2.983825)--(6,1);
   \draw(8.25,0)--(8.75,0.5)--(9.25,0);
   \draw(4.375,3.375)--(5.125,4.125)--(8.75,0.5);
   \node[wleaf]  at (3.5,0) {}; 
   \node[wleaf]  at (4.5,0) {};
   \node[wleaf]  at (7.75,0) {};
   \node[wleaf]  at (8.25,0) {};
   \node[wleaf]  at (9.25,0) {};
   \node[invertex]  at (2.75,1.75) {};
   \node[invertex]  at (4,0.5) {};
   \node[invertex]  at (3.983825,2.983825) {};
   \node[invertex]  at (4.375,3.375) {};
    \node[rootvertex]  at (5.125,4.125) {};
   \node[invertex]  at (8.75,0.5) {};
   \node at (5.5,-1) {\text{ $U_8^{8}$ }};
\end{tikzpicture}
\caption{The minimum size $8$-universal trees.}
\end{figure}

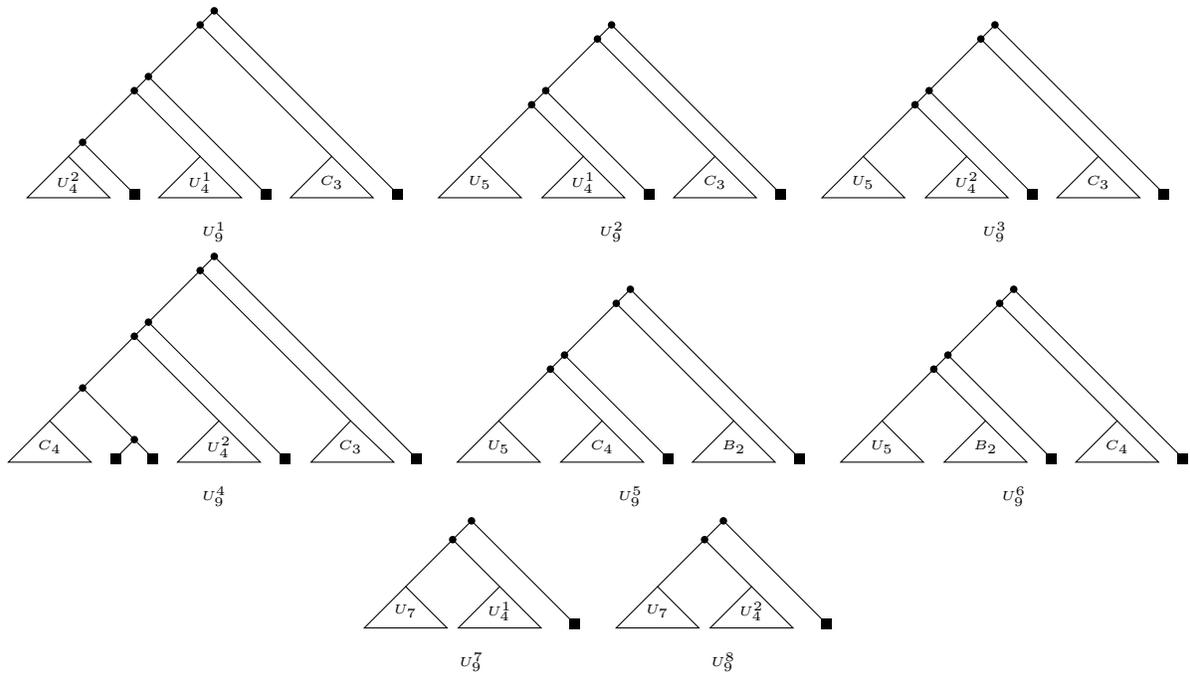
\begin{figure}[H]
\centering
\begin{tikzpicture}[scale=.5,font=\tiny,baseline={(0,0)}]  
   \draw(2,1)--(5.875,4.875)--(10.75,0);
   \draw(2.375,1.375)--(3.75,0);
   \draw(3.75,2.75)--(5.5,1);
   \draw(4.125,3.125)--(7.25,0);
   \draw(5.5,4.5)--(9,1);
   \draw(0.9,-0.1)--(3.1,-0.1)--(2,1)--cycle;
   \node at (2,0.3) {$U_4^2$}; 
   \draw(4.4,-0.1)--(6.6,-0.1)--(5.5,1)--cycle;
   \node at (5.5,0.3) {$U_4^1$}; 
   \draw(7.9,-0.1)--(10.1,-0.1)--(9,1)--cycle;
   \node at (9,0.3) {$C_3$}; 
   \node[wleaf]  at (3.75,0) {}; 
   \node[wleaf]  at (7.25,0) {}; 
   \node[wleaf]  at (10.75,0) {}; 
   \node[invertex]  at (2.375,1.375) {};
   \node[invertex]  at (3.75,2.75) {};
   \node[invertex]  at (4.125,3.125) {};
   \node[invertex]  at (5.5,4.5) {};
   \node[rootvertex]  at (5.875,4.875) {};
\node at (5.875,-1) {\text{ $U_9^{1}$ }};  \end{tikzpicture}
\quad
\begin{tikzpicture}[scale=.5,font=\tiny,baseline={(0,0)}] 
   \draw(2,1)--(5.5,4.5)--(10,0);
   \draw(3.375,2.375)--(4.75,1);
   \draw(3.75,2.75)--(6.5,0);
   \draw(5.125,4.125)--(8.25,1);
   \draw(0.9,-0.1)--(3.1,-0.1)--(2,1)--cycle;
   \node at (2,0.3) {$U_5$};
   \draw(3.65,-0.1)--(5.85,-0.1)--(4.75,1)--cycle;
   \node at (4.75,0.3) {$U_4^1$};
   \draw(7.15,-0.1)--(9.35,-0.1)--(8.25,1)--cycle;
   \node at (8.25,0.3) {$C_3$};
   \node[wleaf]  at (6.5,0) {}; 
   \node[wleaf]  at (10,0) {};  
   \node[invertex]  at (3.375,2.375) {};
   \node[invertex]  at (3.75,2.75) {};    
   \node[invertex]  at (5.125,4.125) {};
   \node[rootvertex]  at (5.5,4.5) {};
\node at (5.5,-1) {\text{ $U_9^{2}$ }};  \end{tikzpicture}
\quad
\begin{tikzpicture}[scale=.5,font=\tiny,baseline={(0,0)}] 
   \draw(2,1)--(5.5,4.5)--(10,0);
   \draw(3.375,2.375)--(4.75,1);
   \draw(3.75,2.75)--(6.5,0);
   \draw(5.125,4.125)--(8.25,1);
   \draw(0.9,-0.1)--(3.1,-0.1)--(2,1)--cycle;
   \node at (2,0.3) {$U_5$};
   \draw(3.65,-0.1)--(5.85,-0.1)--(4.75,1)--cycle;
   \node at (4.75,0.3) {$U_4^2$};
   \draw(7.15,-0.1)--(9.35,-0.1)--(8.25,1)--cycle;
   \node at (8.25,0.3) {$C_3$};
   \node[wleaf]  at (6.5,0) {}; 
   \node[wleaf]  at (10,0) {};  
   \node[invertex]  at (3.375,2.375) {};
   \node[invertex]  at (3.75,2.75) {};    
   \node[invertex]  at (5.125,4.125) {};
   \node[rootvertex]  at (5.5,4.5) {};
\node at (5.5,-1) {\text{ $U_9^{3}$ }};  \end{tikzpicture}
\quad
\begin{tikzpicture}[scale=.5,font=\tiny,baseline={(0,0)}] 
   \draw(2,1)--(6.375,5.375)--(11.75,0);
   \draw(2.875,1.875)--(4.25,0.5);
   \draw(3.75,0)--(4.25,0.5)--(4.75,0);
   \draw(4.25,3.25)--(6.5,1);
   \draw(4.625,3.625)--(8.25,0);
   \draw(6,5)--(10,1);
   \draw(0.9,-0.1)--(3.1,-0.1)--(2,1)--cycle;
   \node at (2,0.3) {$C_4$};
   \draw(5.4,-0.1)--(7.6,-0.1)--(6.5,1)--cycle;
   \node at (6.5,0.3) {$U_4^2$};
   \draw(8.95,-0.1)--(11.15,-0.1)--(10,1)--cycle;
   \node at (10,0.3) {$C_3$};
   \node[wleaf]  at (3.75,0) {}; 
   \node[wleaf]  at (4.75,0) {}; 
   \node[wleaf]  at (8.25,0) {}; 
   \node[wleaf]  at (11.75,0) {}; 
   \node[invertex]  at (2.875,1.875) {};
   \node[invertex]  at (4.25,0.5) {};
   \node[invertex]  at (4.25,3.25) {};
   \node[invertex]  at (4.625,3.625) {};
   \node[invertex]  at (6,5) {};
   \node[rootvertex]  at (6.375,5.375) {};
\node at (6.375,-1) {\text{ $U_9^{4}$ }};  \end{tikzpicture}
\quad
\begin{tikzpicture}[scale=.5,font=\tiny,baseline={(0,0)}] 
   \draw(2,1)--(5.5,4.5)--(10,0);
   \draw(3.375,2.375)--(4.75,1);
   \draw(3.75,2.75)--(6.5,0);
   \draw(5.125,4.125)--(8.25,1);
   \draw(0.9,-0.1)--(3.1,-0.1)--(2,1)--cycle;
   \node at (2,0.3) {$U_5$};
   \draw(3.65,-0.1)--(5.85,-0.1)--(4.75,1)--cycle;
   \node at (4.75,0.3) {$C_4$};
   \draw(7.15,-0.1)--(9.35,-0.1)--(8.25,1)--cycle;
   \node at (8.25,0.3) {$B_2$};
   \node[wleaf]  at (6.5,0) {}; 
   \node[wleaf]  at (10,0) {};  
   \node[invertex]  at (3.375,2.375) {};
   \node[invertex]  at (3.75,2.75) {};    
   \node[invertex]  at (5.125,4.125) {};
   \node[rootvertex]  at (5.5,4.5) {};
\node at (5.5,-1) {\text{ $U_9^{5}$ }};  \end{tikzpicture}
\quad
\begin{tikzpicture}[scale=.5,font=\tiny,baseline={(0,0)}] 
   \draw(2,1)--(5.5,4.5)--(10,0);
   \draw(3.375,2.375)--(4.75,1);
   \draw(3.75,2.75)--(6.5,0);
   \draw(5.125,4.125)--(8.25,1);
   \draw(0.9,-0.1)--(3.1,-0.1)--(2,1)--cycle;
   \node at (2,0.3) {$U_5$};
   \draw(3.65,-0.1)--(5.85,-0.1)--(4.75,1)--cycle;
   \node at (4.75,0.3) {$B_2$};
   \draw(7.15,-0.1)--(9.35,-0.1)--(8.25,1)--cycle;
   \node at (8.25,0.3) {$C_4$};
   \node[wleaf]  at (6.5,0) {}; 
   \node[wleaf]  at (10,0) {};  
   \node[invertex]  at (3.375,2.375) {};
   \node[invertex]  at (3.75,2.75) {};    
   \node[invertex]  at (5.125,4.125) {};
   \node[rootvertex]  at (5.5,4.5) {};
\node at (5.5,-1) {\text{ $U_9^{6}$ }};  \end{tikzpicture}
\quad
\begin{tikzpicture}[scale=.5,font=\tiny] 
\draw (1,1) -- (2.25,2.25) -- (3.5,1);
 \draw (2.25,2.25) -- (2.75,2.75) -- (5.5,0);
 \draw (1,1)--(-0.1,-.1)--(2.1,-.1)--cycle;
 \node at (1,0.35) {$U_7$};  
 \draw(2.4,-.1)--(3.5,1)--(4.6,-.1)--cycle;
 \node at (3.5,0.35) {$U_4^1$};  
 \node[wleaf]  at (5.5,0) {};
 \node[invertex]  at (2.25,2.25) {};
 \node[rootvertex]  at (2.75,2.75) {};
\node at (2.75,-1) {$U_9^7$}; 
\end{tikzpicture}
\quad
\begin{tikzpicture}[scale=.5,font=\tiny] 
\draw (1,1) -- (2.25,2.25) -- (3.5,1);
 \draw (2.25,2.25) -- (2.75,2.75) -- (5.5,0);
 \draw (1,1)--(-0.1,-.1)--(2.1,-.1)--cycle;
 \node at (1,0.35) {$U_7$};  
 \draw(2.4,-.1)--(3.5,1)--(4.6,-.1)--cycle;
 \node at (3.5,0.35) {$U_4^2$};  
 \node[wleaf]  at (5.5,0) {};
 \node[invertex]  at (2.25,2.25) {};
 \node[rootvertex]  at (2.75,2.75) {};
\node at (2.75,-1) {$U_9^8$}; 
\end{tikzpicture}
\caption{The minimum size $9$-universal trees.}
\end{figure}

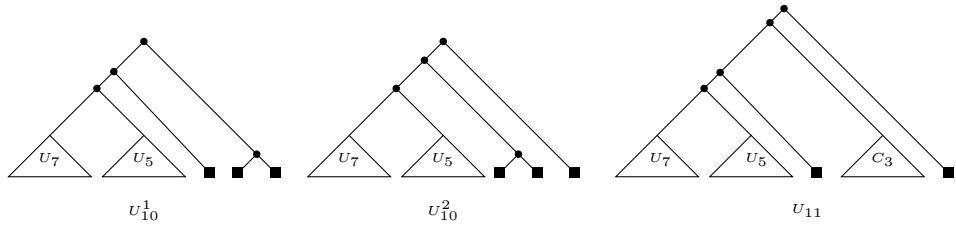
\begin{figure}[H]
   \centering
 \begin{tikzpicture}[scale=.5,font=\tiny,baseline={(0,0)}] 
   \draw(2.25,2.25)--(3.5,1);
   \draw(1,1)--(2.7,2.7);
   \draw(2.7,2.7)--(5.25,0);
   \draw (2.75,2.75) -- (3.5,3.5) -- (7,0);
   \draw (6,0) -- (6.5,0.5);
   \draw (1,1)--(-.1,-.1)--(2.1,-.1)--cycle;
   \draw (3.5,1) -- (2.4,-.1)-- (4.6,-.1)--cycle;
   \node at (1,0.35) {$U_7$};   
   \node[wleaf]  at (5.25,0) {}; 
   \node[wleaf]  at (6,0) {}; 
   \node[wleaf]  at (7,0) {}; 
   \node at (3.5,0.35) {$U_5$};  
   \node[invertex]  at (2.25,2.25) {};
   \node[invertex]  at (2.7,2.7) {}; 
   \node[rootvertex]  at (3.5,3.5) {}; 
   \node[invertex]  at (6.5,0.5) {}; 
\node at (3.5,-1) {$U_{10}^{1}$};  
\end{tikzpicture}\quad
\begin{tikzpicture}[scale=.5,font=\tiny,baseline={(0,0)}] 
    \draw (3,3)--(3.5,3.5)--(7,0);
   \draw (5.5,0.5) -- (6,0);
   \draw (2.25,2.25) -- (3,3) -- (5.5,0.5);
   \draw (1,1) -- (2.25,2.25) -- (3.5,1);
   \draw (5,0) -- (5.5,0.5);
   \draw (1,1)--(-.1,-.1)--(2.1,-.1)--cycle;
   \draw (3.5,1) -- (2.4,-.1)-- (4.6,-.1)--cycle;
   \node at (1,0.35) {$U_7$};   
   \node[wleaf]  at (5,0) {}; 
   \node[wleaf]  at (6,0) {}; 
   \node[wleaf]  at (7,0) {}; 
   \node at (3.5,0.35) {$U_5$};  
   \node[invertex]  at (2.25,2.25) {};
   \node[invertex]  at (3,3) {}; 
   \node[rootvertex]  at (3.5,3.5) {}; 
   \node[invertex]  at (5.5,0.5) {}; 
\node at (3.5,-1) {$U_{10}^{2}$};  
\end{tikzpicture}
\quad
       \begin{tikzpicture}[scale=.5,font=\tiny,baseline={(0,0)}] 
       \draw(1,1)--(2.25,2.25)--(3.5,1);
       \draw(2.25,2.25)--(2.675,2.675)--(5.25,0);
       \draw(2.675,2.675)--(4,4)--(7,1);
       \draw(4,4)--(4.375,4.375)--(8.75,0);
       \draw(-.1,-.1)--(2.1,-.1)--(1,1)--cycle;
       \draw(2.4,-.1)--(4.6,-.1)--(3.5,1)--cycle;
       \draw(5.9,-.1)--(8.1,-.1)--(7,1)--cycle;
       \node at (1.1,0.35) {$U_7$}; 
       \node at (3.65,0.35) {$U_5$}; 
       \node at (7,0.35) {$C_3$}; 
       \node[wleaf]  at (5.25,0) {}; 
       \node[wleaf]  at (8.75,0) {}; 
       \node[invertex]  at (2.25,2.25) {};
       \node[invertex]  at (2.675,2.675) {};
       \node[invertex]  at (4,4) {};
       \node[rootvertex]  at (4.375,4.375) {};
       \node at (5,-1) {\text{ $U_{11}$ }};  
       \end{tikzpicture}
   \caption{The minimum size $k$-universal trees for $k\in\{10,11\}$.}
\end{figure}

\section{Example of a 12-universal tree} \label{sec:12universal}
\begin{figure}[H]
    \centering
     \begin{tikzpicture}[scale=.5,font=\tiny,baseline={(0,0)}] 
       \draw(1,1)--(2.25,2.25)--(3.5,1);
       \draw(2.25,2.25)--(2.675,2.675)--(5.25,0);
       \draw(2.675,2.675)--(4,4)--(7,1);
       \draw(4,4)--(4.375,4.375)--(8.75,0);
       \draw(-.1,-.1)--(2.1,-.1)--(1,1)--cycle;
       \draw(2.4,-.1)--(4.6,-.1)--(3.5,1)--cycle;
       \draw(5.9,-.1)--(8.1,-.1)--(7,1)--cycle;
       \node at (1,0.35) {$U_8^1$}; 
       \node at (3.5,0.35) {$U_6^1$}; 
       \node at (7,0.35) {$C_3$}; 
       \node[wleaf]  at (5.25,0) {}; 
       \node[wleaf]  at (8.75,0) {}; 
       \node[invertex]  at (2.25,2.25) {};
       \node[invertex]  at (2.675,2.675) {};
       \node[invertex]  at (4,4) {};
       \node[rootvertex]  at (4.375,4.375) {};
       \end{tikzpicture}
    \caption{A $12$-universal tree of size 28.}
    \label{fig:12universal}
\end{figure}
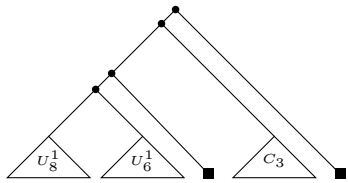

\end{document}